\documentclass[12pt, a4paper]{amsart}
\usepackage{amsmath,amssymb,amscd,amsfonts}
\newtheorem{theorem}{Theorem}[section]
\newtheorem{rem}[theorem]{Remark}
\newtheorem{ex}[theorem]{Example}
\newtheorem{defn}[theorem]{Definition}
\newtheorem{lemma}[theorem]{Lemma}

\newtheorem{proposition}[theorem]{Proposition}
\newtheorem{conjecture}[theorem]{Conjecture}
\newtheorem{corollary}[theorem]{Corollary}

\def\Ricci{\mathop{\rm Ricci}\nolimits}
\def\Vol{\mathop{\rm Vol}\nolimits}
\def\Id{\mathop{\rm Id}\nolimits}

\def\Supp{\mathop{\rm Supp}\nolimits}
\def\Aut{\mathop{\rm Aut}\nolimits}
\def\rk{\mathop{\rm rk}\nolimits}

\def\dbar{\overline\partial}
\def\ddbar{\partial\overline\partial}
\def\cO{{\mathcal O}}

\def\cJ{{\mathcal J}}
\def\cA{{\mathcal A}}
\def\cR{{\mathcal R}}

\def\cP{{\mathcal P}}
\def\cF{{\mathcal F}}

\def\cY{{\mathcal Y}}
\def\cD{{\mathcal D}}
\let\ol=\overline
\let\ep=\varepsilon

\let\wh=\widehat

\def\bC{{\mathbb C}}
\def\bR{{\mathbb R}}

\def\bD{{\mathbb D}}
\def\bG{{\mathbb G}}
\def\bP{{\mathbb P}}
\def\bB{{\mathbb B}}

\begin{document}
\title[Parabolic Riemann]{Value Distribution Theory for \\
Parabolic Riemann Surfaces}

\author{Mihai P\u aun, Nessim Sibony}

\address{Mihai P\u aun \endgraf
Korea Institute for Advanced Study, \endgraf
Seoul, 130-722
South KOREA}
\email{paun@kias.re.kr}

\address{Nessim Sibony \endgraf   
 Univ. Paris-Sud, Universit\'e Paris-Saclay, 91405 Orsay,
France\endgraf 
and\endgraf 
Korea Institute for Advanced Study, \endgraf
Seoul, 130-722
South KOREA}
\email{nessim.sibony@math.u-psud.fr}

\maketitle

\begin{abstract}  A conjecture by Green-Griffiths states that if $X$ is a projective manifold of general type, then there exists an algebraic proper subvariety of $X$ which contains the image of all holomorphic curves from the complex plane to $X$. To our knowledge, the general case is far from being settled. We question here the choice of the complex plane as a source space.

Let ${\mathcal Y}$ be a parabolic Riemann surface, i.e bounded subharmonic functions defined on 
${\mathcal Y}$ are constant. The results of Nevanlinna's theory for holomorphic 
maps $f$ from ${\mathcal Y}$ to the projective line are parallel to the classical case when ${\mathcal Y}$ is the complex line except for a term involving a weighted Euler characteristic. Parabolic Riemann surfaces could be hyperbolic in the Kobayashi sense.

Let $X$ be a manifold of general type, and let $A$ be an ample line bundle on $X$. It is known that there exists a holomorphic jet differential $P$ (of order $k$) with values in the dual of A. If the map $f$ has infinite area and if ${\mathcal Y}$ has finite Euler characteristic, then $f$ satisfies the differential relation induced by $P$. As a consequence, we obtain a generalization of Bloch Theorem concerning the Zariski closure of maps $f$ with values in a complex torus.
An interesting corollary of these techniques is a refined Ax-Lindemann theorem to transcendental affine varieties (the classical case concerns affine algebraic varieties) for which we give a proof.
We then study the degree of Nevanlinna's currents $T[f]$ associated to a parabolic leaf of a foliation ${\mathcal F}$ by Riemann surfaces on a compact complex manifold. We show that the degree of $T[f]$ on the tangent bundle of the foliation is bounded from below in terms of the counting function of $f$ with respect to the singularities of ${\mathcal F}$, and the Euler characteristic of 
${\mathcal Y}$. In the case of complex surfaces of general type, we obtain a complete analogue of McQuillan's result: a parabolic curve of infinite area and finite Euler characteristic tangent to ${\mathcal F}$ is not Zariski dense.That requires some analysis of the dynamics of foliations by Riemann Surfaces.

\end{abstract}

\tableofcontents

\section{Introduction}

 Let $X$ be a compact complex manifold. S. Kobayashi 
 introduced a pseudo-distance, determined by
 the complex structure of $X$. We recall here its infinitesimal version,
 cf. \cite{Koba}.
 
Given a point $x\in X$ and a tangent vector $\displaystyle v\in T_{X, x}$ 
at $X$ in $x$, the length of $v$  with respect to the \emph{Kobayashi-Royden
pseudo-metric} is the following quantity

$${\bf k}_{X, x}(v):= \inf\{ \lambda> 0 ; \exists f: {\mathbb D}\to X, 
f(0)= x, \lambda
f^\prime(0)= v\},$$
where ${\mathbb D}\subset {\mathbb C}$ is the unit disk, and $f$ is a holomorphic map.
\smallskip

\noindent We remark that it may very well happen that 
${\bf k}_{X, x}(v)= 0$; however, thanks to Brody re-parametrization lemma,
this situation has a geometric counterpart, as follows. If there exists a couple
$(x, v)$ as above such that $v\neq 0$ and such that ${\bf k}_{X, x}(v)=  0$, then one can construct a holomorphic 
non-constant map $f : {\mathbb C}\to X$. The point $x$ is not necessarily in 
the image of $f$.

In conclusion, if any entire curve drawn on $X$ is constant, then 
the pseudo-distance defined above is a distance, and 
we say that $X$ is \emph{Brody hyperbolic}, 
or simply \emph{hyperbolic} (since
most of the time we will be concerned with compact manifolds). 
 \smallskip
 
 \noindent As a starting point for the questions with which we will be concerned with in this article, 
 we have the following result.
 
 \begin{proposition} \label{p1}
   Let $\Omega\subset X$ be a Kobayashi hyperbolic
   open set, which is hyperbolically embedded in a
   compact complex manifold $X$. Let ${\mathcal C}$ be a Riemann surface. 
Let $E\subset {\mathcal C}$ be a closed, countable set. Then any holomorphic map 
$f: {\mathcal C}\setminus E\to \Omega$ admits a (holomorphic) extension
$\widetilde f: {\mathcal C}\to X$.
\end{proposition}
In particular, in the case of the complex plane we infer that any holomorphic map $f: {\mathbb C}\setminus E\to \Omega$
must be the restriction of an application $\widetilde f: \bP^1\to X$ (under the hypothesis of Proposition 1.1). We will give a proof and discuss some related statements and
questions in the first paragraph of this paper. Observe however that if the cardinal of $E$ is at least 2, then $\bC\setminus E$ is Kobayashi hyperbolic.

Our next remark is that the surface ${\mathbb C}\setminus E$ is a particular case of a 
{\sl parabolic Riemann surface}; we recall here the definition.
A Riemann surface ${\mathcal Y}$ is parabolic if any bounded subharmonic
 function defined on ${\mathcal Y}$ is constant. This is a large class of surfaces, including e.g.  $Y\setminus \Lambda$, where $Y$ is a compact Riemann surface of arbitrary genus and
 $\Lambda \subset Y$ 
 is any closed polar set. It is known
 (cf. \cite{Ahlfors}, \cite{Stoll}, page 80) that a 
 non-compact Riemann surface $\mathcal Y$ is \emph{parabolic} if and only if it admits a smooth exhaustion function 
 $$\sigma: {\mathcal Y}\to [1, \infty[$$ 
 such that: 
 
 
 \noindent $\bullet$ $\tau:= \log \sigma $ is harmonic in the complement of a compact set of 
 ${\mathcal Y}$. Moreover, we impose the normalization 
 $$\int_{\cY}dd^c\log \sigma
= 1,\leqno(1)$$ 
where the operator $d^c$ is defined as follows
$$d^c:= \frac{\sqrt{-1}}{4\pi}(\overline \partial- \partial).$$
 \medskip
 

On the boundary $S(r):= (\sigma =r)$ of the parabolic ball of radius $r$ we have the induced measure
$$d\mu_r:= d^c\log \sigma|_{S(r)}.$$
The measure $d\mu_r$ has total mass equal to 1, by the relation (1) combined with 
Stokes formula.

Since we are dealing with general parabolic surfaces, the growth of the Euler characteristic of the balls
$\bB(r)= (\sigma< r)$ will appear very often in our estimates. We introduce the following notion.

\begin{defn} Let $(\cY, \sigma)$ be a parabolic Riemann surface, together with an exhaustion function
as above. For each $t\geq 1$ 
such that $S(t)$ is non-singular we denote
by $\chi_\sigma(t)$ the Euler characteristic of the domain $\bB(t)$, and let
$$\mathfrak{X}_\sigma (r):= \int_1^r\big|\chi_\sigma(t)\big|\frac{dt}{t}$$
be the (weighted) mean Euler characteristic of the ball of radius $r$.
\end{defn}

If 
$\cY= \bC$, then $\mathfrak{X}_\sigma (r)$ is bounded by $\log r$. The same type of bound is verified if $\cY$ is the complement of a finite number of points in $\bC$. If $\cY= \bC\setminus E$ where $E$ is a closed polar set of infinite cardinality, then things are more subtle, depending on the density of the distribution of the points of $E$
in the complex plane. However,
an immediate observation is that the surface $\cY$ has finite Euler characteristic if and only if 
$$\mathfrak{X}_\sigma (r)= \cO(\log r).\leqno(2)$$


\medskip


\noindent In the first part of this article we will extend a few classical results in hyperbolicity theory
to the context of parabolic 
Riemann surfaces, as follows.

We will review  
the so-called ``first main theorem" and the logarithmic derivative lemma 
for maps $f:\cY\to X$, where $X$ is a compact complex manifold. We also give a version of the first main theorem 
with respect to an ideal $\cJ\subset \cO_X$. This will be a convenient language when studying foliations with singularities.

\noindent As a consequence, we derive a vanishing result for jet differentials, similar to
the one obtained in case $\cY= \bC$, as follows.

Let $\mathcal P$ be a jet differential of order $k$ and degree $m$ on $X$, with values in the dual of an ample bundle (see \cite{Dem2}; we recall a few basic facts about this notion in the next section). 
Then we prove the following result.

\begin{theorem}\label{t1}
Let $\mathcal Y$ be a parabolic Riemann surface. We consider a holomorphic map
$f: \mathcal Y\to X$ such that we have
$$\lim\sup_{r\to \infty} \frac {\mathfrak{X}_\sigma (r)}{T_{f, \omega}(r)}= 0.
\leqno{(\dagger)}$$
Let $\cP$ be an invariant jet differential of order $k$ and degree $m$, with values in the dual of an ample line bundle. Then we have 
$$\mathcal P\left(j_k(f)\right)= 0$$
identically on 
$\mathcal Y$.\end{theorem}
For example, the requirement above is satisfied if $\cY$ has finite Euler characteristic and infinite area.
In the previous statement we denote by $j_k(f)$ the $k^{\rm th}$ jet associated to the map $f$.
If $\mathcal Y= \mathbb C$, then this result is well-known, starting with the seminal work of A. Bloch
cf. \cite{Bloch}; see also \cite{Dem1}, \cite{Siu1} and the references therein, in particular the work of T. Ochiai \cite{Och}, Green-Griffiths \cite{GrG} and Y. Kawamata \cite{Kawa}.
It is extremely useful in the investigation of the hyperbolicity properties of projective manifolds. 
In this context, the above result says that the vanishing result still holds in the context of Riemann surfaces of 
(possibly) infinite Euler characteristic, provided that the growth of this topological invariant is
slow when compared to $T_{f, \omega}(r)$. It also holds when the source is the unit disk, provided that the growth is large enough. 
\medskip

\noindent As a consequence of Theorem \ref{t1} we obtain the following result (see section 4, Corollary 4.6). 
Let $X$ be a projective manifold, and let $D= Y_1+\dots +Y_N$ 
be an effective snc (i.e. simple normal crossings) divisor. We assume that there exists a logarithmic 
jet differential $\cP$ on $(X, D)$ with values in a bundle $A^{-1}$, where $A$ is ample. 
Let $f:\bC\to X$ be an entire curve which do not satisfies the differential equation defined by $\cP$. 
Then we obtain a lower bound for the number of intersection points of $f(\bD_r)$ with $D$ as $r\to\infty$, where 
$\bD_r\subset \bC$ is the disk of radius $r$.

\medskip

\noindent Concerning the existence of jet differentials, we recall Theorem 0.1 in \cite{Dem3}, see also 
\cite{Merker}.

\begin{theorem}
\label{t2}
Let $X$ be a manifold of general type. Then there is a couple of integers
$m\gg k\gg 0$ and a (non-zero) holomorphic invariant
jet differential $\mathcal P$ of order $k$ and degree $m$ with values in the dual of an ample line bundle
$A$.
\end{theorem}
Thus, our result \ref{t1} can be used in the context of the 
general type manifolds. 
  
\medskip

\noindent As a consequence of Theorem \ref{t1}, we obtain the following analogue of Bloch's theorem. It does not seem to be possible to derive this result by using e.g. Ahlfors-Schwarz negative curvature arguments. Observe also that we cannot use a Brody-Green type argument, because the Brody reparametrization lemma is not available in 
our context.  
\begin{theorem}
\label{bloch}
Let ${\mathbb C}^N/\Lambda$ be a complex torus, and let $\cY$ be a parabolic 
Riemann surface. We consider a holomorphic map $$f: \cY\to {\mathbb C}^N/\Lambda$$ which verifies the condition $(\dagger)$.
Then the smallest analytic subset $X$
containing the closure of the image of $f$ 
is either the translate of a sub-torus in 
${\mathbb C}^N/\Lambda$, or there exists
a map $\cR: X\to W$ onto a general type subvariety of an abelian variety $W\subset A$ such that the area of the curve $\cR\circ f$ is finite.
\end{theorem} 

\medskip

\noindent In the second part of this paper our aim is to recast some of the work of M.~McQuillan and M. Brunella concerning the Green-Griffiths conjecture in the parabolic setting. We first recall 
the statement of this problem.

\begin{conjecture}{\rm (\cite{GrG})}
Let $X$ be a projective manifold of general type. Then there exists an algebraic subvariety 
$W\subsetneq X$ which contains the image of all holomorphic curves $f: \bC\to X$. 
\end{conjecture}
\noindent It is hard to believe that this conjecture is correct for 
manifolds $X$ of dimension $\geq 3$. On the other hand, it is very likely that this holds true for
surfaces (i.e. $\dim X= 2$), on the behalf of the results available in this case.
\smallskip
   
Given a map $f:\cY\to X$ defined on a parabolic Riemann surface $\mathcal Y$,
we can associate a Nevanlinna-type closed positive current $T[f]$. If $X$ is a surface of general type and if $\cY$ has finite Euler characteristic, then there exists an integer 
$k$ such that the $k$-jet of $f$ satisfies an algebraic relation. As a consequence, there exists a 
foliation $\cF$ by Riemann surfaces on the space of $k$-jets $X_k$ 
of $X_0$, such that the lift of $f$
is tangent to $\cF$. In conclusion we are naturally led to consider the 
pairs $(X, \cF)$, where $X$ is a compact manifold, and
$\cF$ is a foliation by curves on $X$. We denote by $T_{\cF}$ the 
so-called tangent bundle of $\cF$.

We derive a lower bound of the intersection 
number $\displaystyle \int_XT[f]\wedge c_1(T_{\cF})$ in terms of 
a Nevanlinna-type counting function of the intersection of $f$ with the
singular points of $\cF$.  
As a consequence, if $X$ is a complex surface and $\cF$ has reduced singularities,
we show that $\displaystyle \int_XT[f]\wedge c_1(T_{\cF})\geq 0$. For this part we follow closely the original argument of \cite{McQ}. 

When combined with a result by Y. Miyaoka, the preceding inequality shows that the classes
$\{T[f]\}$ and $c_1(T_\cF)$ are orthogonal. Since the class of the current $T[f]$ is nef, we show 
by a direct argument that we have
$\displaystyle \int_X\{T[f]\}^2= 0$, and from this we infer that the Lelong numbers of the diffuse part $R$ of $T[f]$ are equal to zero at each point of $X$.

This regularity property of $R$ is crucial, since it allows to show --via the Baum-Bott formula
and an elementary fact from dynamics--
that we have 
$\displaystyle \int_XT[f]\wedge c_1(N_{\cF})\geq 0,$ where $N_{\cF}$ is the normal 
bundle of the foliation, and $c_1(N_{\cF})$ is the first Chern class of $N_{\cF}$.  
\medskip

\noindent We then obtain the next result, in the spirit of \cite{McQ}.
\begin{theorem}\label{c1}
Let $X$ be a surface of general type, and consider a holomorphic map $f:\cY\to X$, where $\cY$ is a parabolic 
Riemann surface sich that $\mathfrak{X}_\sigma (r)= o(T_f(r))$. 
We assume that $f$ is tangent to a holomorphic foliation $\cF$; then the dimension of the Zariski 
closure of $f(\cY)$ is at most 1.
\end{theorem}

\smallskip

\noindent  In the last section of our survey we 
give a short proof of M. Brunella index theorem \cite{Brun1}. 
Furthermore, we show that 
that this important result admits the following generalization.
 
Let $L$ be a line bundle on a complex surface $X$, such that  
$S^mT^\star_X\otimes L$ has a non-identically zero section $u$. 
Let $f$ be a holomorphic map from a parabolic Riemann surface $\cY$ to $X$, directed by the multi-foliation $\cF$
defined by $u$, i.e. we have $u\left((f^{\prime})^{\otimes m}\right)= 0$
Then we show that we have
$$\int_Xc_1(L)\wedge T[f]\geq 0.$$
Brunella's theorem corresponds to the case $m=1$ and $L= N_{\cF}$: indeed, a
foliation on $X$ can be seen as a section of $T^\star_X\otimes N_{\cF}$
(or in a dual manner, as a section of $T_X\otimes T_{\cF}^\star$).

If $X$ is a minimal surface of general type, such that $c_1^2> c_2$, we see that
this implies Theorem \ref{c1} directly, i.e. without considering the $T[f]$-degree of 
the tangent of $\cF$. In particular, we do not need to invoke Miyaoka's generic 
semi-positivity theorem, nor the blow-up procedure of McQuillan.

\noindent It is a very interesting problem to generalize the inequality
above in the framework of higher order jet differentials, cf.\,section 7 for a precise statement.
\medskip




\section{Preliminaries} 

\medskip

\subsection{Motivation: an extension result}

\smallskip

\noindent We first give the proof of Proposition 1.1. We refer to \cite{Koba} for further results in this direction. 
\smallskip

\noindent {\sl Proof} (of Proposition 1.1) A first observation is that it is enough to deal with the
case where $E$ is a single point. 
Indeed, assume that this case is settled. We consider the set $E_0\subset E$ such that 
the map $f$ does not extends across $E_0$; our goal is to prove that we have
$E_0= \emptyset$. If this is not the case, then we remark that $E_0$ contains at least an
isolated point --since it is countable, closed and non-empty--, and thus we obtain a contradiction. 


Thus we can assume that we have a holomorphic map 
$$f: {\mathbb D}^\star\to \Omega$$
where ${\mathbb D}^\star$ is the pointed unit disc. Let $g_{\mathcal P}$ be the Kobayashi metric on $\Omega$; we remark that by hypothesis, $g_{\mathcal P}$ is non-degenerate. By the {\sl distance decreasing property} of this metric we infer that
$$\vert f^\prime(t)\vert^2_{g_{\mathcal P}}\leq \frac{1}{\vert t\vert ^2\log^2\vert t\vert ^2}$$
for any $t\in {\mathbb D}^\star$. On the other hand,
$\Omega$ is hyperbolically embedded in $X$, so (by definition of hyperbolic embedding) we have
$\displaystyle \vert f^\prime(t)\vert^2_{g_X}\leq C \vert f^\prime(t)\vert^2_{g_{\mathcal P}}$. Here $g_X$ is a metric on $X$ and $C$ is a positive constant (depending on the metric) independent of the point $t\in {\mathbb D}^\star$. 
We thus have
$$\vert f^\prime(t)\vert^2_{g_X}\leq \frac{C}{\vert t\vert ^2\log^2\vert t\vert ^2}.
\leqno(3)$$
This is a crucial information, since now we can argue as follows. The inequality 
(3) implies that the area of the graph associated to our map $\Gamma_f^0\subset \bD^\star \times X$
defined by
$$\Gamma_f^0:= \{(t, x)\in \bD^\star \times X : f(t)= x\}$$
is finite. By the theorem of Bishop-Skoda (cf. \cite{Sibony} and the references therein) this implies that there exists an analytic subset $\Gamma\subset \bD \times X$ whose restriction to $\bD^\star \times X$ is precisely $\Gamma_f^0$. Hence we infer that the fiber of the projection 
$\Gamma\to \bD$ on the second factor is a point. Indeed, if this is not the case,
then the area of the image (via $f$) of the disk of radius $\varepsilon$ is bounded from below by a constant independent of $\varepsilon> 0$. This of course cannot happen,
as one can see by integrating the inequality (3) over the disk of radius $\varepsilon$.

\qed

\medskip

\noindent In connection with this result, we recall the following conjecture proposed in \cite{Noguchi}.

\begin{conjecture} Let $X$ be a Kobayashi hyperbolic compact manifold of dimension $n$. 
We denote by ${\mathbb B}$ the unit ball in $\mathbb C^p$, and let $E$ be a closed pluripolar subset of 
${\mathbb B}$. Then any holomorphic map
$$f: {\mathbb B}\setminus E\to X$$
extends across $E$.

\end{conjecture}

\noindent In the case where $X$ is the quotient of a bounded domain in $\mathbb C^n$, a proof of this conjecture was proposed by M. Suzuki in \cite{Suzuki}. In general, even if we assume that the holomorphic bisectional curvature 
of $X$ is bounded from above by -1, the conjecture above seems to be open.
\medskip

\noindent In the same spirit, we quote next a classical result due to A. Borel.
Let $\cD\subset {\mathbb C}^n$ be a bounded symmetric domain, and let
$\Gamma\subset \Aut(\cD)$ be an arithmetically defined torsion-free group, cf. \cite{Borel}
for definitions and references. The quotient $\Omega:= \cD/\Gamma$ admits a projective
compactification, say $X$ called the \emph{Baily-Borel compactification}. Using Proposition 1,1, the result in \cite{Borel}
can be stated as follows

\begin{theorem}\cite{Borel}
  Let $\phi: C\setminus E\to \cD/\Gamma$ be a holomorphic non-constant map, where $E$ is a countable. Then $\varphi$ extends across $E$ and defines a
  holomorphic map $\widetilde\phi: C\to X$.
\end{theorem}   

\subsection{Jet spaces}
We will recall here a few basic facts concerning the jet spaces 
associated to complex manifolds; we refer to \cite {Dem2}, \cite{Koba} for a more complete overview.
\smallskip

\noindent Let $X$ be an $l$-dimensional complex space. We denote by $J^k(X)$ the space of $k$-jets of holomorphic discs, described as follows. Let $f$ and $g$ be two germs of analytic discs $({\mathbb C}, 0)\to (X, x)$, we say that they define {the same $k$ jet at $x$} if their derivatives at zero coincide up to order $k$, i. e.
$$f^{(j)}(0)= g^{(j)}(0)$$
for $j= 0,...k$. The equivalence classes defined by this equivalence relation 
is denoted by $J^k(X,x)$; as a set, $J^k(X)$ is the union of $J^k(X, x)$
for all $x\in X$. We remark that if $x\in X_{\rm reg}$ is a non-singular point of 
$X$, then $J^k(X, x)$ is isomorphic to ${\mathbb C}^{kl}$, via the identification
$$f\to \big(f^\prime(0),..., f^{(k)}(0)\big).$$
This map is not intrinsic, it depends on the choice of some local coordinate system needed to express the derivatives above; at a global level the projection map
$$J^k(X_{\rm reg})\to X_{\rm reg}$$
is a holomorphic fiber bundle (which is not a vector bundle in general, since the transition functions are 
polynomial instead of linear). 

If $k= 1$, and $x\in X_{\rm reg}$ is a regular point, then $J^1(X, x)$ is the tangent space of $X$ at $x$.
We also mention here that 
the structure of the analytic space $J^k(X)$
at a singular point of $X$ is far more complicated. 

We assume that $X$ is a subset of a complex manifold $M$; then for each positive integer $k$ we have a natural inclusion
$$J^k(X)\subset J^k(M)$$
and one can see that the space $J^k(X)$ is the Zariski closure of the 
analytic space $J^k(X_{\rm reg})$ in the complex manifold $J^k(M)$ (note that this coincides with
the topological closure)
\smallskip

\noindent Next we recall the definition of the main geometric objects we will use in the analysis of the 
structure of the subvarieties  of complex tori which are 
Zariski closure of some parabolic image.

As before, let $x\in X_{\rm reg}$ be a regular point of $X$; we consider a coordinate system
$(x^1,..., x^l)$ of $X$ centered at $x$. We consider the symbols
$$dx^1,..., dx^l, d^2x^1,..., d^2x^l,..., d^kx^1,..., d^kx^l$$
and we say that the weight of the symbol $d^px^r$ is equal to $p$, for any $r=1,..., l$. A 
\emph{jet differential} of order $k$ and degree $m$ at $x$ is a homogeneous polynomial 
of degree $m$ in 
$\displaystyle (d^px^r)_{p=1,...,k, r= 1,..., l}$; we denote by $E_{k, m}^{\rm GG}(X, x)$ the vector space of all such polynomials,
We denote the set
$$E_{k, m}^{\rm GG}\big(X_{\rm reg}\big):= \cup _{x\in X_{\rm reg}}E_{k, m}^{\rm GG}(X, x)$$
has a structure of vector bundle, whose global sections are called jet differentials of weight 
$m$ and order $k$.
A global section $\mathcal P$ of the bundle $E_{k, m}^{\rm GG}(X)$ can be written locally as 
$$\mathcal P= \sum_{\vert \alpha_1\vert+\dots+ k\vert \alpha_k\vert= m}a_\alpha (dx)^{\alpha_1}
\dots (d^kx)^{\alpha_k} ;$$
here we use the standard multi-index notation. 

Let $f: (\mathbb C, 0)\to (X, x)$ be a $k$-jet at $x$. The group $\bG_k$ of $k$--jets of biholomorphisms of $(\bC, 0)$ acts on $J_k(X)$,
and
we say that the 
operator $\mathcal P$ is {\sl invariant} if 
$$\mathcal P\big((f\circ \varphi)^\prime,\dots, (f\circ \varphi)^{(k)}\big)= {\varphi^\prime}^{m}
\mathcal P\big(f^\prime,\dots, f^{(k)}\big). $$
The bundle of invariant jet differentials is
denoted by $E_{k, m}(X)$. We will recall next an alternative description of this bundle, which will be very useful in what follows.

Along 
the next few lines, we indicate a compactification of the quotient
$J_k^{reg}(X)/\bG_k$ following \cite{Dem2}, where $J_k^{reg}(X)$ denote the space of non-constants jets.

We start with the pair $(X, V)$, where $V\subset T_X$ is a subbundle of the tangent 
space of $X$. Then we define $X_{1}:= \bP(T_X)$, and the bundle $\displaystyle 
V_1\subset T_{X_{1}}$ is defined fiberwise by

$$V_{1, (x, [v])}:= \{\xi\in T_{X_{1} (x, [v])} : d\pi(\xi)\in \bC v \}$$

\noindent where $\pi: X_1\to X$ is the canonical projection and 
$v\in V$. It is easy to
see that we have the following parallel description of $V_1$.
Consider a non-constant disk
$u:(\bC, 0)\to (X, x)$. We can lift it to $\displaystyle X_{1}$ 
and denote the resulting germ by 
$u_1$. Then the derivative of $u_1$ belongs to the $V_1$ directions. 

In a more formal manner, we have the exact sequence
$$0\to T_{X_1/X}\to V_1 \to \cO_{X_1}(-1)\to 0$$
where $\displaystyle \cO_{X_1}(-1)$ is the tautological bundle on $X_1$, and 
$\displaystyle T_{X_1/X}$ is the relative tangent bundle corresponding to the fibration 
$\pi$. This shows that the rank of $V_1$ is equal to the rank of $V$.

Inductively by this procedure we get a tower of manifolds $\displaystyle 
(X_{k}, V_k)$, starting from $(X, T_X)$ and it turns out
that we have an embedding $J_k^{reg}/\bG_k\to X_{k}$. On each manifold $X_{k}$, 
we have a 
tautological bundle $\displaystyle \cO_{X_k}(-1)$, and the positivity of its dual plays an important 
role here. 

We denote by $\pi_k: X_k\to X$ the projection, and consider the direct image sheaf
$\pi_{k*}\bigl(\cO_{X_k}(m)\bigr)$. The result is a vector bundle $E_{k, m}(X)$
whose sections are precisely the invariant jet differentials considered above.

The fiber of $\pi_k$ at a non-singular point
of $X$ is denoted by $\cR_{n, k}$; it is a rational manifold, and it is a compactification 
of the quotient 
$\bC^{nk}\setminus 0/ \bG_k$.

\smallskip

The articles \cite{Bloch}, \cite{Dem2}, \cite{Siu1} (to quote only a few) show that the existence of 
jet differentials are crucial in the analysis of the entire maps 
$f: \bC\to X$. 
As we will see in the next sections, 
they play a similar role in the study of the images of the parabolic Riemann surfaces. 

\section{Basics of Nevanlinna Theory for Parabolic Riemann Surfaces}

Let $\mathcal Y$ be a parabolic Riemann surface; as we have recalled in the introduction, this means that
there exists 
a smooth
 exhaustion function 
 $$\sigma: {\mathcal Y}\to [1, \infty[$$ 
 such that: 
 
 
 \noindent $\bullet$ The function $\tau:= \log \sigma$ is harmonic in the complement of a compact subset of 
 ${\mathcal Y}$, and we have $\int_{\cY}dd^c\tau= 1$.
 \medskip
 
 \noindent We denote by $\mathbb B(r)\subset \mathcal Y$ the parabolic ball of radius $r$, that is to say
 $$\mathbb B(r):= \{y\in \mathcal Y : \sigma(y)\leq r\}.$$
 For almost every value $r\in \mathbb R$, the sphere $S(r):= \partial \mathbb B(r)$ is a smooth curve drawn on $\mathcal Y$. The induced length measure on $S(r)$ is
 equal to 
 $$d\mu_r:= d^c\log \sigma|_{S(r)}.$$
 
Let $v: \mathcal Y\to [-\infty, \infty[$ be a function defined on $\mathcal Y$, 
such that locally near every point of $\cY$ it can be written as a difference of two 
subharmonic functions, i.e. $dd^cv$ is of order zero.
\medskip

\noindent Then we recall here the following formula, which will be very useful in what follows.

\begin{proposition} \label{p3}
{\rm (Jensen formula)} For every $r\geq 1$ large enough we have
\begin{equation} 
\begin{split}
\int_{1}^r\frac{dt}{t}\int_{\mathbb B(t)}dd^cv= & \int_{S(r)}vd\mu_r- \int_{\mathbb B(r_0)}vdd^c\tau=\\
= & \int_{S(r)}vd\mu_r+ \cO(1)
\end{split}
\nonumber
\end{equation}

where we have $\tau= \log \sigma$.
\end{proposition}

\medskip

\proof The arguments are standard. To start with, we remark that for each 
regular value $r$ of $\sigma$ the function $v$ is integrable with respect to 
the measure $d\mu_r$ over the sphere $S(r)$. Next, we have

\begin{equation} 
\begin{split}
\int_{1}^r\frac{dt}{t}\int_{\mathbb B(t)}dd^cv= & 
\int_{\bB(r)}\big(\log{r}- \log{\sigma}\big) dd^cv= \\
= & \int\log^+\frac{r}{\sigma}dd^cv = \int v dd^c\big(\log^+\frac{r}{\sigma}\big)\\
= & \int_{S(r)}vd\mu_r- \int_{\mathbb B(r)}vdd^c\tau.
\end{split}
\nonumber
\end{equation}
\qed

\begin{rem}
{\rm As we can see, the Jensen formula above holds true even without the assumption that the 
function $\tau$ is harmonic outside a compact set. The only difference is eventually as $r\to \infty$, since the term $\displaystyle \int_{\mathbb B(r)}
vdd^c\tau$ may tend to infinity.}
\end{rem}

\medskip

We reformulate next the notion of mean Euler characteristic in analytic terms. To this end,
we first recall that 
the tangent bundle $T_{\mathcal Y}$ of a non-compact parabolic surface admits a trivializing global holomorphic section 
$v\in H^0({\mathcal Y}, T_{\mathcal Y})$, cf. \cite{Gunning} (actually, any such Riemann surface 
admits a submersion into $\bC$). Using the Poincar\'e-Hopf index theorem, 
we obtain the following result.

\begin{proposition} \label{p2}
Let $(\cY, \sigma)$ be a parabolic Riemann surface, so that $\log \sigma$ is harmonic
in the complement of a compact set. Then we have
$$\mathfrak{X}_\sigma (r)= {1\over 2}\int_{S(r)}\log^+ \vert d\sigma(v)\vert^2 d\mu_r +\cO(\log r)$$
for any $r\geq 1$.
\end{proposition}

\begin{proof} Since $v$ is a vector of type (1,0), Jensen formula gives

\begin{equation} 
\begin{split}
 \int_{S(r)}\log^+\vert d\sigma(v)\vert^2 d\mu_r  =  & \int_{S(r)}\log^+\vert \partial \sigma(v)\vert^2 d\mu_r  \\
= & 2\log r+ \int_{S(r)}\log^+\vert \partial\log\sigma(v)\vert^2 d\mu_r \\
 = & \int_{1}^r\frac{dt}{t}\int_{\mathbb B(t)}dd^c \log^+ |\partial \log \sigma(v)|^2+ 2\log r+ \cO(1). 
\end{split}
\nonumber
\end{equation} 
Observe that we are using the fact that the function $\partial\log \sigma(v)$ is holomorphic 
in the complement of a compact set, so $\log^+ |\partial \log \sigma(v)|^2$ is subharmonic.

The term $\cO(\log r)$ above depends on the exhaustion function $\sigma$ and on the fixed vector field $v$, but the quantity 
$$\int_{1}^r\frac{dt}{t}\int_{\mathbb B(t)}dd^c \log |\partial \log\sigma(v)|^2$$
is equal to the weighted Euler characteristic $\mathfrak{X}_\sigma (r)$ of the domains ${\mathbb B(r)}$, in particular it is 
independent of $v$, up to a bounded term. This can be seen as a consequence of the 
Poincar\'e-Hopf index theorem, combined with the fact that the function 
$\partial_{v}\log\sigma$ is holomorphic, so $\displaystyle dd^c\log |\partial_{v}\log \sigma|^2$ count the critical points 
of $\partial_{v} \sigma $.
\end{proof}

\noindent We obtain next the {\sl first main theorem} and the 
{\sl logarithmic derivative lemma} of Nevanlinna theory in the parabolic setting. The results are 
variations on
well-known techniques (see \cite{Stoll} and the references therein). But for the reader's convenience we will reproduce here the arguments. 

\smallskip

Let $X$ be a compact complex manifold, and let $L\to X$ be a line bundle on $X$, endowed with a smooth metric $h$. We make no particular assumptions concerning the curvature form 
$\Theta_h(L).$
Let $s$ be a non-trivial section of $L$ normalized such that 
$\sup_X\vert s\vert= 1$, and let 
$f: \mathcal Y\to X$ be a holomorphic map, where $\mathcal Y$ is parabolic. 

We define the usual 
characteristic function of $f$ with respect to $\Theta_h(L)$ as follows
$$T_{f, \Theta_h(L)}(r):= \int_{r_0}^r\frac{dt}{t}\int_{\mathbb B(t)}f^\star \Theta_h(L).$$
If the form $\Theta_h(L)$ is positive definite, then precisely as in the classical case $\cY= \bC$, the area of the image of $f$
will be finite if and only if $T_{f, \Theta_h(L)}(r)= \cO(\log r)$ as $r\to \infty$.

Let 
$$N_{f, s}(r):= \int_{r_0}^rn_{f, s}(t)\frac{dt}{t}$$
be the counting function, where $n_{f, s}(t)$ is the number of zeroes of $s\circ f$ in the parabolic ball of
radius $t$ (counted with multiplicities). Hence we assume implicitly that the image of 
$f$ is not contained in the set $(s= 0)$.
Moreover, in our context the proximity function becomes
$$m_{f, s}(r):= \frac{1}{2\pi}\int_{S(r)} \log \frac{1}{\vert s\circ f\vert_h} d\mu_r.$$
In the important case of a (meromorphic) function $F: \mathcal Y\to \mathbb P^1$, one usually takes $L:= \cO(1)$, hence $\Theta_h(L)$ is the Fubini-Study metric, and $s$ the section vanishing at infinity. The proximity function
becomes 
$$m_{f, \infty}(r):= \frac{1}{2\pi}\int_{S(r)} \log_+ \vert f\vert d\mu_r,\leqno (4)$$
where $F:= [f_0: f_1]$, $f= f_1/f_0$ and $\log_+:= \max (\log, 0)$.
\medskip

\noindent As a consequence of Jensen formula, we derive the next result.

\begin{theorem}
\label{FMT}
With the above notations, we have
$$T_{f, \Theta_h(L)}(r)= N_{f, s}(r)+ m_{f, s}(r)+
\cO(1)\leqno (5)$$
as $r\to\infty$.
\end{theorem}

\proof The argument is similar to the usual one: we apply Jensen's formula cf.\ Proposition \ref{p3} to the function $v:= \log \vert s\circ f\vert_h$. Recall that
 that the Poincar\'e-Lelong equation gives $dd^c\log \vert s\vert_h^2= [s=0]- \Theta_h(L)$, which
 implies
$$dd^cv= \sum_jm_j\delta_{a_j}- f^\star(\Theta_h(L))$$
so by integration we obtain (5). \qed
\medskip

\begin{rem} {\rm If the measure $dd^c\tau$ does not have a compact support, then the term $\cO(1)$
in the equality (5) is to be replaced by 
$$\int_{\bB(r)}\log \vert s\circ f\vert^2_{h}dd^c\tau.\leqno (6)$$
We observe that, thanks to the normalization condition we impose to $s$, the term (6) is negative. In particular we infer that 
$$T_{f, \Theta_h(L)}(r)\geq N_{f, s}(r)+ \int_{\bB(r)}\log \vert s\circ f\vert^2_{h}dd^c\tau\leqno (7)$$
for any $r\geq r_{0}.$
}
\end{rem}
\medskip

We will discuss now a version of Theorem \ref{FMT} which will be very useful in dealing with singular foliations.
Let $\cJ\subset \cO_X$ be a coherent ideal of holomorphic functions. We consider a 
finite covering of 
$X$ with coordinate open sets $(U_\alpha)_{\alpha\in \Lambda}$, such that on $U_\alpha$ the ideal $\cJ$ is generated by the 
holomorphic functions $\displaystyle (g_{\alpha i})_{i=1\dots N_\alpha}$. 

Then we can construct a function $\psi_{\cJ}$, such that for each $\alpha \in \Lambda$ the difference 
$$\psi_{\cJ}- \log(\sum_i |g_{\alpha i}|^2)\leqno(8)$$
is bounded on $U_\alpha$. Indeed, let $\rho_\alpha$ be a partition of unity subordinated to $(U_\alpha)_{\alpha\in \Lambda}$. We define the function $\psi_{\cJ}$ as follows
$$\psi_{\cJ}:= \sum_\alpha \rho_\alpha\log(\sum_i |g_{\alpha i}|^2)\leqno(9)$$
and the boundedness condition (8) is verified, since there exists a constant $C>0$ such that
$$C^{-1}\leq \frac{\sum_i |g_{\alpha i}|^2}{\sum_i |g_{\beta i}|^2}\leq C$$
holds on $U_\alpha\cap U_\beta$, for each pair of indexes $\alpha, \beta$. 
In the preceding context, the function $\psi_{\cJ}$ corresponds to $\log |s|_h^2$.

We can define a counting function and a proximity function for a holomorphic map
$f:\cY\to X$ with respect to the analytic set defined by $\cJ$, as follows. Let $(t_j)\subset \cY$
be the set of solutions of the equation 
$$\exp\big(\psi_{\cJ}\circ f(y)\big)= 0.$$
For each $r>0$ we can write
$$\psi_{\cJ}\circ f(y)|_{\bB(r)}= \sum_{\sigma(t_j)< r}\nu_j\log|y-t_j|^2 + \cO(1)$$
for a set of multiplicities $\nu_j$, and then the counting function is defined as follows
$$N_{f, \cJ}(r):= \sum_j \nu_j\log\frac{r}{\sigma(t_j)}.\leqno(10)$$
In a similar way, the proximity function is defined as follows
$$m_{f, \cJ}(r):= -\int_{S(r)}\psi_{\cJ}\circ f d\mu_r.\leqno(11)$$
 
Since $\cJ$ is coherent, the principalization theorem (cf. e.g. \cite{Kollar}), there exists a non-singular manifold $\wh X$ together with a birational map
$p: \wh X\to X$ such that the inverse image of the ideal $\cJ$ is equal to $\cO_{\wh X}(-D)$, where 
$D:= \sum_je_jW_j$ is a simple normal crossing divisor on $\wh X$. Recall that $\cO_{\wh X}(-D)\subset \cO_{\wh X}$ is the sheaf of holomorphic functions vanishing on $D$.
In terms of the function 
$\psi_{\cJ}$ associated to the ideal $\cJ$, this can be expressed as follows
$$\psi_{\cJ}\circ p= \sum_je_j\log|s_{j}|_{h_j}^2+ \theta\leqno(12)$$
where $W_j= (s_j=0)$, the metric $h_j$ on $\cO(W_j)$ is arbitrary (and non-singular), and 
where $\theta$ is a bounded function on $\wh X$.

\noindent Since we assume that the image of the map $f$ is not contained in the 
zero set of the ideal $\cJ$, we can define the lift $\wh f: \cY\to \wh X$ of $f$ to $\wh X$ such that 
$p\circ \wh f= f$. We have the next result.

\begin{theorem}\label{t4}
Let $\cO(D)$ be the line bundle associated to the divisor $D$; we endow it with the metric induced 
by $(h_j)$, and let $\Theta_D$ be the associated curvature form. Then we have
$$T_{\wh f, \Theta_D}(r)= N_{f, \cJ}(r)+ m_{f, \cJ}(r)+ \cO(1)$$
as $r\to \infty$.
\end{theorem}
\noindent The argument is completely similar to the one given for Theorem \ref{FMT} (by using the relation (12)). We basically apply the first main theorem for each $s_j$ and add up the contributions.\qed

\medskip
We will treat now another important result in Nevanlinna theory, namely the \emph{logarithmic derivative 
lemma} in the parabolic context. To this end, we will suppose that as part of the data we are given
a vector field
$$\xi\in H^0(\mathcal Y, T_{\mathcal Y})$$
which is nowhere vanishing and hence it trivializes the tangent bundle of our surface $\mathcal Y$.
We denote by 
$f^\prime$ the section $df(\xi)$ of the bundle $f^\star T_X$.
For example, if $\mathcal Y= \mathbb C$, then we can take $\displaystyle \xi= \frac{\partial}{\partial z}$.

In the proof of the next result, we will need the following form of the co-area formula.
Let $\psi$ be a 1-form defined on the surface $\mathcal Y$; then we have
$$\int_{\mathbb B(r)}d\sigma \wedge \psi= \int_0^rdt\int_{S(t)}\psi \leqno (13)$$
for any $r> 1$. In particular, we get
$$\frac{d}{dr}\int_{\mathbb B(r)}d\sigma \wedge \psi= \int_{S(r)}\psi \leqno (14)$$

\noindent We have the following version of the classical logarithmic derivative lemma. If $\cY$ is a pointed disk, this was established by J. Noguchi in \cite{Noguchi_1} (whom we thank for the reference).

\begin{theorem}\label{t5}
Let $f: \mathcal Y\to \mathbb P^1$ be a meromorphic map defined on a parabolic Riemann surface $\mathcal Y$. The inequality
$$m_{{f^\prime}/{f}, \infty}(r)\leq C\big(\log T_{f}(r)+ \log r\big)+ \mathfrak{X}_\sigma(r)$$
holds true for all $r$ outside a set of finite Lebesgue measure. We also get a similar estimate for higher order derivatives.
\end{theorem}

\proof Within the framework of Nevanlinna theory, this kind of results can be derived in many ways if
$\mathcal Y$ is the complex plane; the proof presented here follows an argument due to Selberg in \cite{Selberg}.

On the complex plane $\mathbb C\subset \mathbb P^1$ we consider the coordinate $w$ corresponding to $[1 : w]$ in homogeneous coordinates on $\mathbb P^1$. The 
form
$$\Omega:= \frac{1}{\vert w\vert^2(1+ \log^4\vert w\vert)}
\frac{\sqrt -1}{2\pi}dw\wedge d\overline w\leqno(15)$$
on $\mathbb C$ has finite volume, as one can easily check by a direct computation.
In what follows, we will use the same letter to denote the expression of the meromorphic function 
$\mathcal Y\to \mathbb C$ induced by $f$. For each $t\geq 0$, we denote by $n(t, f, w)$ the number of zeroes of the function $z\to f(z)- w$ in the parabolic ball $\mathbb B(t)$ and we have
$$\int_{\mathbb B(t)}f^\star \Omega= \int_{\mathbb C}n(t, f, w) \Omega\leqno (16)$$
by the change of variables formula. 

Next, by integrating the relation (16) above and using Theorem \ref{FMT}, we infer the following 
$$\int_{1}^r\frac{dt}{t}\int_{\mathbb B(t)}f^\star \Omega= 
\int_{\mathbb C}N_{f-w, \infty}(r) \Omega\leq \int_{\mathbb C}T_f(r) \Omega + 
\int_{\mathbb C}\log^+|w|\Omega;\leqno (17)$$
by Remark 3.5. This last quantity is smaller than
$\displaystyle C_0T_f(r),$ 
where $C_0$ is a positive constant. Thus, we have
$$\int_{1}^r\frac{dt}{t}\int_{\mathbb B(t)}f^\star \Omega\leq C_0T_f(r).\leqno (18)$$

A simple algebraic computation shows the next inequality
\begin{equation} 
\begin{split}
 \log\Big(1+ \frac{\vert df(\xi)\vert ^2}{\vert f(z)\vert ^2}\Big)\leq & 
\log\Big(1+ \frac{\vert df(\xi)\vert ^2}{\vert f(z)\vert ^2(1+ \log^2\vert f(z)\vert ^2)|d\sigma(\xi)|^2}\Big)+ \\
+ & \log(1+ \log^2\vert f(z)\vert ^2) 
+ \log(1+ |d\sigma(\xi)|^2)
\end{split}
\nonumber
\end{equation}
and therefore we get

\begin{equation} 
\begin{split}
m_{{f^\prime}/{f}, \infty}(r)\leq & \frac{1}{4\pi}\int_{S(r)}\log\Big(1+ \frac{\vert df(\xi)\vert ^2}{\vert f(z)\vert ^2}\Big)d\mu_r
\leq \\
\leq &  \log^{+} \int_{S(r)}\frac{\vert df(\xi)\vert ^2}{\vert f(z)\vert ^2(1+ \log^2\vert f(z)\vert)}
\frac{1}{\vert d\sigma_z(\xi)\vert^2}d\mu_r+ \\
+ & \int_{S(r)}\log(1+ \log^2\vert f(z)\vert ^2)d\mu_r+\\
+ &  \frac{1}{4\pi}\int_{S(r)}\log(1+ \vert d\sigma_z(\xi)\vert^2)d\mu_r+ C_{1},
\end{split}
\nonumber
\end{equation}
where $C_{1}$ is a positive constant; here we used the concavity of the log function.

By the formula (14) we obtain
\begin{equation} 
\begin{split}
\int_{S(r)} & \frac{\vert df(\xi)\vert ^2}{\vert f(z)\vert ^2(1+ \log^2\vert f(z)\vert )}
\frac{1}{\vert d\sigma_z(\xi)\vert^2}d\mu_r= \\
= \frac{1}{r}\frac{d}{dr}\int _{\mathbb B(r)} & 
\frac{\vert df(\xi)\vert ^2}{\vert f(z)\vert ^2(1+ \log^2\vert f(z)\vert )}
\frac{1}{\vert d\sigma_z(\xi)\vert^2}d\sigma\wedge d^c\sigma.
\end{split}
\nonumber
\end{equation}

Next we show that we have 
$$\frac{\vert df(\xi)\vert ^2}{\vert f(z)\vert ^2(1+ \log^2\vert f(z)\vert)}
\frac{1}{\vert d\sigma_z(\xi)\vert^2}d\sigma\wedge d^c\sigma= f^\star \Omega.$$
Indeed this is clear, since we can choose a local coordinate $z$ such that $\displaystyle \xi= \frac{\partial}{\partial z}$
and moreover we have $\displaystyle d\sigma\wedge d^c\sigma= \frac{\sqrt -1}{2\pi}
\Big\vert\frac {\partial \sigma}{ \partial z}\Big\vert^2dz\wedge d\overline z$ (we are using here the fact that
$f$ is holomorphic).
\smallskip

Let $H$ be a positive, strictly increasing function defined on $(0, \infty)$. It is immediate to check that the set of numbers $s\in \bR_+$ such that the inequality 
$$H^\prime(s)\leq H^{1+ \delta}(s)$$
is not verified, is of finite Lebesgue measure. By applying this calculus lemma to the function
$\displaystyle H(r):= \int_{1}^r\frac{dt}{t}\int_{\mathbb B(t)}f^\star \Omega$
we obtain
\begin{equation} 
\begin{split}
\log^+ \frac{1}{r}\frac{d}{dr}\int _{\mathbb B(r)}f^\star \Omega\leq & \log^+ \frac{1}{r} 
\Big(\int _{\mathbb B(r)}f^\star \Omega\Big)^{1+ \delta}+ \cO(1)\\
\leq & \log^+\Big( r^\delta \Big[\frac{d}{dr}\int_{1}^r\frac{dt}{t}\int_{\mathbb B(t)}f^\star \Omega\Big]^{1+\delta }\Big)+ \cO(1)\\
\leq & \delta \log r + \frac{1}{2}\log^+\Big(\int_{1}^r\frac{dt}{t}\int_{\mathbb B(t)}f^\star \Omega\Big)^{(1+\delta)^2}+ \cO(1)
\end{split}
\nonumber
\end{equation}
for all $r$ outside a set of finite measure.

The term  
$$\int_{S(r)}\log(1+ \log^2\vert f(z)\vert ^2)d\mu_r$$
is bounded, up to a constant, by $\log T_{f}(r)$; combined with Proposition \ref{p2}, this  
implies the desired inequality.
\qed 

\smallskip

\noindent It is a simple matter to deduce the so-called \emph{second main theorem} of Nevanlinna theory 
starting from the logarithmic derivative lemma (cf. e.g. \cite{Dem1}). The parabolic version of this result can be stated as follows.

\begin{theorem}\label{t6}
Let $f: ({\mathcal Y}, \sigma)\to \mathbb P^1$ be a meromorphic function.
We denote by $\displaystyle N_{R_f}(r)$ the Nevanlinna counting function for the ramification 
divisor associated to $f$. Also, we
use the classical notation
$$\displaystyle \delta_f(a):= \underline\lim_r\frac{m_{f, a}(r)}{T_{f, \omega}(r)}= 1- \overline\lim_r
\frac{N_{a}(r)}{T_{f, \omega}(r)}.$$
Then for any set of distinct points $(a_{j})_{1\leq j\leq p}$ in $\bP^1$ there exists a set 
$\Lambda\subset \bR_+$ of finite Lebesgue measure such that
$$N_{R_f}(r)+ \sum_{j=1}^pm_{f, a_j}(r)\leq 2T_{f, \omega}(r)+ \mathfrak{X}_\sigma(r)+ \cO\big(\log r+ \log^+T_{f, \omega}(r)\big)+ 
$$
for all $r\in \bR_+\setminus \Lambda$.
\end{theorem}

\noindent For the proof we refer e.g. to \cite{Dem1}; as we have already mentioned, it is a direct consequence of the 
logarithmic derivative lemma. As a consequence, we have
$$\sum_{j=1}^p\delta_f(a_j)\leq 2+ \overline\lim_{r\to \infty}\frac{\mathfrak{X}_\sigma(r)}{T_{f, \omega}(r)}.$$


\section{The Vanishing Theorem}

Let $\cP$ be an invariant jet differential of order $k$ and degree $m$. We assume that it has values in
the dual of an ample line bundle, that is to say
$$\cP\in H^0(X, E_{k, m}T^\star_X\otimes A^{-1})$$
where $A$ is an ample line bundle on $X$. 

Let $f: \mathcal Y\to X$ be a parabolic curve on $X$; assume that we are given the exhaustion function
$\sigma$ and a vector field $\xi$ such that we have
$$\overline\lim_{r\to \infty}\frac{\mathfrak{X}_\sigma(r)}{T_{f, \omega}(r)}= 0\leqno(19)$$
on $\cY$.
As we have already recalled in the preliminaries, the operator $\cP$ can be seen as 
section of $\displaystyle \cO_{X_k}(m)$ on $X_k$. 

On the other hand, the curve $f$ admits a canonical lift to $X_k$ as follows. One first observes that
the derivative $df: T_{\cY}\to f^\star T_X$ induces a map 
$$f_1: \cY\to \bP(T_X).$$
We remark that to do so we do not need any supplementary data, since $df(v_1)$ and $df(v_2)$
are proportional, provided that $v_j\in T_{\cY, t}$ are tangent vectors at the same point. Using the notations of section
two, it turns out that the 
curve $f_1$ is tangent to $\displaystyle V_1\subset T_{X_1}$, so that we can continue this procedure and define 
inductively $f_k: \cY\to X_k$. 
\medskip

\noindent We prove next the following result.

\begin{theorem} \label{t7}
We assume that condition (19) holds. Then the image of $f_k$ is contained in the 
zero set of the section of $\cO_{X_k}(m)\otimes A^{-1}$ defined by the jet differential $\cP$.
\end{theorem}

\proof We observe that we have
$$df_{k-1}(\xi): \cY\to f_k^\star\big(\cO_{X_k}(-1)\big)$$
that is to say, the derivative of $f_{k-1}$ is a section of the inverse image of the tautological bundle. 
Thus the quantity 
$$\cP\big(df_{k-1}(\xi)^{\otimes m}\big)$$
is a section of $f_k^{\star}(A^{-1})$, where the above notation means that we are evaluating
$\cP$ at the point $f_{k}(t)$ on the $m^{\rm th}$ power of the section above at $t\in \cY$.

As a consequence, if $\omega_A$ is the curvature form of $A$, we have
$$\sqrt{-1}\ddbar \log \vert \cP\big(df_{k-1}(\xi)^{\otimes m}\big)\vert ^2\geq f_k^{\star}(\omega_A).\leqno(20)$$
The missing term involves the Dirac masses at the critical points of 
$f$. We observe that the positivity of the bundle $A$ is fully used at this point: we obtain an upper bound for the characteristic function of $f$.
By integrating and using Jensen formula, we infer that we have
$$\int_{S(r)}\log \vert \cP\big(df_{k-1}(\xi)^{\otimes m}\big)\vert ^2d\mu_r\geq T_{f, \omega_A}(r)+ \cO(1)
$$
as $r\to \infty$.  

Now we follow the arguments in \cite{Dem1}. There exists a finite set of rational functions 
$u_j: X_k\to \bP^1$ and a positive constant $C$ such that we have
$$\log^+ \vert \cP\big(df_{k-1}(\xi)^{\otimes m}\big)\vert ^2\leq C\sum_j
\log^+ \frac{\vert d(u_j\circ f_{k-1})(\xi)\vert ^2}{\vert u_j\circ f_{k-1}\vert^2 }$$
pointwise on $\cY$. 
Indeed, we can use the meromorphic functions $u_j$ as local coordinates on $X$, and then
we can write the jet differential $\cP$ as $Q\big(f, d^p(\log u_j\circ f)\big)$, hence the previous inequality.
We invoke next the logarithmic derivative lemma (Theorem \ref{t5}) 
established in the previous section, and so we infer that we have
$$\int_{S(r)}\log \vert \cP\big(df_{k-1}(\xi)^{\otimes m}\big)\vert ^2d\mu_r\leq
C(\log T_{f_{k-1}(r)}+ \log r+ \mathfrak{X}_\sigma(r))\leqno (21)$$
out of a set of finite Lebesgue measure.
It is not difficult to see that the characteristic function corresponding to $f_{k-1}$ is the smaller than 
$CT_{f}(r)$, for some constant $C$; by combining the relations (20) and (21) we have
$$T_{f}(r)\leq C\big(\log T_{f}(r)+ \log r+ \mathfrak{X}_\sigma(r)\big)$$
in the complement of a set of finite Lebesgue measure.
Since by 
assumption $\mathfrak{X}_\sigma(r)+ \log(r)= o\big(T_f(r)\big)$, we get a contradiction. Therefore, if the image of $f_k$ is not contained in the zero set of 
$\cP$, then the area of $f$ is finite. \qed
\medskip

\noindent Let $E\subset \bC$ be a polar subset of the complex plane. 
In the case 
$$\cY= \bC\setminus E,$$ 
we show that we have the following version of the 
previous result in the context of arbitrary jet differentials.

\begin{theorem}\label{t8}
Let $f: \bC\setminus E\to X$ be a holomorphic curve; we assume that the area of $f$ is infinite, and that 
condition (19) is satisfied. Then $\cP(f^\prime,\dots, f^{(k)})\equiv 0$ for any holomorphic jet differential 
$\cP$ of degree $m$ and order $k$ with values in the dual of an ample line bundle.    
\end{theorem}

\proof The argument is similar to the proof of the preceding 
Theorem \ref{t7}, except that we use the pointwise
inequality 
$$\log^+ \vert \cP(f^\prime,\dots, f^{(k)})\vert\leq C\sum_{j}\sum_{l=1}^k\log \vert d^l\log (u_j\circ f)\vert $$
combined with Theorem \ref{t5} in order to derive a contradiction. \qed
\medskip

\noindent The following statement is an immediate consequence of Theorem \ref{t7}.

\begin{corollary}\label{ample}
Let $X$ be a projective manifold whose cotangent bundle is ample. Then $X$ does not admits any 
holomorphic curve $f:\cY\to X$ with infinite area such that that the 
condition (19) is satisfied.
\end{corollary}

We recall that in the articles \cite{Brot}, \cite{Xie}, \cite{BD} it is shown that
any generic complete intersection $X$ of sufficiently high degree and codimension in $\bP^n$
satisfies the hypothesis of Theorem \ref{ample}.

\medskip

\noindent To state our next result, we consider the following data. Let $X$ be a non-singular, projective manifold
and let $D= Y_1+\dots+ Y_l$ be an effective divisor, such that the pair
$(X, D)$ is log-smooth (this last condition means that the hypersurfaces $Y_j$ are non-singular, and that they have transverse intersections). In some cases we have
$$H^0\big(X, E_{k, m}{T}^\star_X\langle D\rangle\otimes A^{-1}\big)\neq 0\leqno (22)$$
where $E_{k, m}{T}^\star_X\langle D\rangle$ is the log version of the space of invariant jet differentials 
of order $k$ and degree $m$. Roughly speaking, the sections of the bundle in (22) are homogeneous 
polynomials in 
$$d^p\log z_{1},\dots ,d^p\log z_{d},  d^p z_{d+1},\dots, d^p z_{n}$$ 
where $p=1,\dots k$ and $z_{1}z_{2}\dots z_{d}= 0$ is a local equation of the divisor $D$.

\medskip

\noindent We have the following result, which is a more general version of Theorem \ref{t8}.

\begin{theorem}\label{t9}
Let $f:\bC\setminus E\to X\setminus D$ be a non-algebraic, holomorphic map. 
If the parabolic Riemann surface $\cY:= \bC\setminus E$ verifies
$$\mathfrak{X}_\sigma(r)= o\big(T_f(r)\big)$$ 
then for any invariant log-jet differential
$$\cP\in H^0\big(X, E_{k, m}{T}^\star_X\langle D\rangle\otimes A^{-1}\big)$$
we have $\cP(f^\prime,\dots, f^{(k)})\equiv 0$.
\end{theorem}

\proof We only have to notice that the "logarithmic derivative lemma"
type argument used in the proof of
the vanishing theorem is still valid in our context, despite of the fact that the jet differential has poles 
along $D$
(see \cite{Ru} for a complete treatment).
Thus, the result follows as above. \qed
\smallskip

\noindent As a corollary of Theorem \ref{t9}, we obtain the following statement.

\begin{corollary}\label{c2}
 Let $(X, D)$ be a pair as above, and let 
$f:\bC\to X$ be a non-algebraic, holomorphic map; we define $E:= f^{-1}(D)$. 
Moreover, we assume that 
there exists an invariant log-jet differential $\cP$ with values in $A^{-1}$ such that
$\cP(f^\prime,\dots, f^{(k)})$ is not identically zero. We denote by $N_E(r)$ the Nevanlinna counting function 
associated to $E$, i.e. $\displaystyle N_E(r)= \int_0^r\frac {dt}{t}{\rm card}\big(E\cap (\sigma< r)\big)$.
Then there exists a constant $C>0$ such that we have 
$$\lim\inf_{r}\frac{N_E(r)}{T_{f}(r)}> C.\leqno (23)$$
Moreover, the constant $C$ is independent of $f$.
\end{corollary}

\begin{rem}{\rm
Observe that \emph{we are not} in the situation of Theorem 4.3, since the image of $f$ can intersect the support of the divisor $D$. Actually the main point in the previous statement is to analyze the intersection of $f$ with $D$. 
}
\end{rem}

\begin{proof} Let $\sigma$ denote the parabolic exhaustion associated to $E$ as in example 2 above, i.e. 
By the proof of the vanishing theorem 
we infer the existence of a constant $C> 0$ such that the inequality
$$T_{f}(r)\leq C^{-1}\big(\log T_{f}(r)+ \cO(\log r)+ \mathfrak{X}_\sigma(r)\big)$$ 
holds as $r\to\infty$ for any map $f$ such that 
$\cP(f^\prime,\dots, f^{(k)})$ is not identically zero. On the other hand, we have $\mathfrak{X}_\sigma(r)= N_E(r)$,
so the relation (23) is verified. The fact that $C$ is independent of $f$ can be seen directly from the previous 
inequality.
\end{proof}

\medskip

\begin{rem}\label{rem}{\rm
  It turns out that there is another important class of Riemann surfaces for which we can
  establish a vanishing theorem of type \ref{t8}. Let $f:\bD\to X$ be a holomorphic map defined on a unit disk, which satisfies the growth condition
  $\displaystyle \sup_{r< 1}\frac{T(f,r)}{\log\frac{1}{1-r}}= \infty$. Then {the
  vanishing theorem \ref{t8} holds for $f$, i.e. given any polynomial}
$$\cP\in H^0\big(X, E_{k, m}{T}^\star_X\otimes A^{-1}\big)$$
  we have $\cP(f^\prime,\dots, f^{(k)})\equiv 0$. The proof is obtained by the same
  arguments as \ref{t8}, given the following result due to M. Tsuji \cite{Tsuji},
  page 222.
  \begin{theorem}\label{tsuji_log}
    Let $\psi:\bD\to \bP^1$ be a holomorphic map. We have
    $$m\left(\frac{\psi^\prime}{\psi}, r\right)= \cO\left(\log\frac{1}{1-r}\right)+ \cO\left(\log T(r)\right)$$    
    for all $r< 1$ except a set of intervals $(I_m)$ such that $\displaystyle
    \int_{\cup_m I_m}\frac{dr}{1-r}<\infty$.
\end{theorem}    
}\end{rem}

\noindent The previous remark \ref{rem} motivates the following ad-hoc notion.

\begin{defn}\label{admiss}
  A Riemann surface $\cY$ is called admissible if it satisfies one of the following two requirements.
  \begin{enumerate}

  \item[(a)] $\cY$ is parabolic.
    \smallskip

  \item[(b)] There exists a compact complex manifold $X$ and a holomorphic map
    $\psi:\bD\to X$ such that $\psi(\bD)\subset \cY\subset X$ and such that
    $$\sup_{r< 1}\frac{T(\psi, r)}{\log\frac{1}{1-r}}= \infty.$$
  \end{enumerate}  
\end{defn}  
\smallskip

\noindent We notice that the term $\displaystyle \log\frac{1}{1-r}$ has a clear interpretation: it is the 
characteristic function of the identity map $\bD\to (\bD, \omega_P)$, where $\omega_P$ is the Poincar\'e metric on the disk $\bD$. Thus the hypothesis (b) means that the growth order of $\psi$ is much faster 
compared to the weighted area of $\bD$ measured with respect to the Poincar\'e metric. Observe that when $X$ is Kobayashi hyperbolic, there are no maps satisfying condition (b).

\medskip

\subsection{A few examples} At the end of this paragraph we will discuss some examples of parabolic surfaces; we will try to emphasize in particular the properties of
the function $\mathfrak{X}_\sigma(r)$.
\smallskip

\noindent (1) Let $E\subset \bC$ be a finite subset of the complex plane. 
Define 
$$\log \sigma:= \log^+|z|+ \sum_{a\in E}\log^+\frac{r}{|z-a|}$$
for $r> 0$ small enough. Then 
clearly we have $\mathfrak{X}_\sigma(r)= \cO(\log r)$ for $\cY:= \bC\setminus E$.

\smallskip

\noindent (2) We treat next the case of $\cY:= \bC\setminus E$, where $E= (a_{j})_{j\geq 1}$
is a closed, countable set of points in $\bC$. As we will see, in this case it is natural to use the Jensen 
formula without assuming that the support of the measure $dd^c\log \sigma$ is compact, see 
Remark 3.2.

Let $(r_{j})_{{j\geq 1}}$ be a sequence of positive real numbers, such that the Euclidean disks 
$\bD(a_{j}, r_{j})$ are disjoints. As in the preceding example, we define the exhaustion function $\sigma$ such that 
$$\log \sigma=  \log^+ \vert z\vert+ \sum_{j\geq 1}
\log^+\frac{r_j}{\vert z- a_{j}\vert};$$
 the difference here is that $d\mu_r$ is no longer a probability measure. 
 However, we have the following inequality
 $$\frac{1}{\sigma}\Big\vert \frac{\partial \sigma}{\partial z}\Big\vert\leq 
\Big(\sum_{j\geq 1}\frac{1}{\vert z- a_{j}\vert}\chi_{\bD(a_{i}, 1)}+ 
\frac{1}{\vert z\vert}\chi_{(\vert z\vert > 1)}\Big)$$
which holds true on $\cY$.
On the parabolic sphere $S(r)$ we have $\vert z\vert < r$ and 
$\displaystyle \vert z- a_j\vert 
> \frac{r_j}{r}$ so that we obtain 
\begin{equation} 
\begin{split}
\mathfrak{X}_\sigma (r) \leq & \log r+ \!\!\!\sum_{\sigma(a_j+ r_je^{i\theta_j})< r}
\log \frac{r}{a_j}\leq \\
\leq & 
\log r+ \sum_{\vert a_j\vert < r}
\log \frac{r}{a_j}: = \log r+ N_{(a_j)_{j\geq 1}}(r).
\end{split}
\nonumber
\end{equation}
Therefore, we can bound the Euler characteristic by the 
counting function for ${(a_j)_{j\geq 1}}$.
\smallskip

\begin{rem} {\rm Let $\cY$ be a parabolic Riemann surface, with an exhaustion
$\sigma$ normalized as in (1).
In more abstract terms, the quantity $\mathfrak{X}_\sigma (r)$ can be estimated along the following lines, by using the same kind of techniques as in the 
proof of the logarithmic derivative lemma.
\begin{equation} 
\begin{split}
\mathfrak{X}_\sigma (r) = & \int_{S(r)}\log^+\Big\vert \frac{\partial \sigma}{\partial z}\Big\vert^2d\mu_r\leq \\
\leq & \frac{2}{\varepsilon}\log^+\int_{S(r)}\Big\vert \frac{\partial \sigma}{\partial z}\Big\vert^\varepsilon d^c\log \sigma =  
\frac{2}{\varepsilon}\log^+\frac{1}{r}\int_{S(r)}\Big\vert \frac{\partial \sigma}{\partial z}\Big\vert^\varepsilon d^c\sigma\leq \\
\leq & \frac{2}{\varepsilon}\log^+\frac{1}{r}
\frac{d}{dr}\int_{B(r)}\Big\vert \frac{\partial \sigma}{\partial z}\Big\vert^\varepsilon d\sigma\wedge d^c\sigma \leq \\
\leq & \frac{2}{\varepsilon}\log^+\frac{1}{r}
\Big(\int_{B(r)}\Big\vert \frac{\partial \sigma}{\partial z}\Big\vert^{2+ \varepsilon} 
dz\wedge d\overline z\Big)^{1+ \delta} \leq \\
\leq & \frac{2}{\varepsilon}\log^+{r^\delta}
\Big(\int_{B(r)}\frac{1}{r}\Big\vert \frac{\partial \sigma}{\partial z}\Big\vert^{2+ \varepsilon} 
dz\wedge d\overline z\Big)^{1+ \delta} \leq \\
\leq & \frac{2\delta}{\varepsilon}\log^+{r}+ \frac{(1+\delta)^2}{\varepsilon}
\log^+
\Big(\int_1^r\frac{dt}{t}\int_{B(t)}\Big\vert \frac{\partial \sigma}{\partial z}\Big\vert^{2+ \varepsilon} 
dz\wedge d\overline z\Big).
\end{split}
\nonumber
\end{equation}
}
\end{rem}
On the other hand, we remark that we have
$$\int_{\bB(t)}d\sigma\wedge d^c\sigma= t^2; \leqno(24)$$
this can be verified e.g. by considering the derivative of the left hand side of the expression (24) above.
However, this does not means that the surface $\cY$ is of finite mean Euler characteristic, since 
we cannot take $\varepsilon= 0$ in our previous computations--and as the example 
(2) above shows it, there is a good reason to that.   
\medskip

\section{Bloch Theorem}

Let $T$ be a complex torus, and let 
$X\subset T$ be an irreducible analytic set. In some sense, the birational geometry of $X$ was completely understood since the work of Ueno \cite{Ueno}. He has established the following result.

\begin{theorem}\label{t10}
\cite{Ueno} Let $X$ be a subvariety of a complex torus $T$. Then there exist a complex torus $T_1\subset T$, a projective variety $W$ and an abelian variety $A$ such that 

\begin{enumerate}

\item We have $W\subset A$ and $W$ is a variety of general type;
\smallskip

\item There exists a dominant (reduction) map $\cR: X\to W$ whose general fiber is isomorphic to $T_1$.

\end{enumerate}
\end{theorem}
 
\noindent Thus, via this theorem the study of an arbitrary submanifold $X$ of a torus is reduced to the 
case where $X$ is of general type.
\smallskip

We analyze next the case where $X$ is the Zariski closure of a {\sl parabolic Riemann surface}, or more generally, of an admissible Riemann surface in the sense of the Definition \ref{admiss}. We recall that the condition $(\dagger)$ on page 4 means that 
the area of the surface is infinite and $\displaystyle \mathfrak{X}_\sigma(r)= {\sl o}(T_f(r))$.
The following result can be seen as the complete analogue of the
classical theorem of Bloch \cite{Bloch}.

\begin{theorem}\label{t11} Let $X\subset T$ be a submanifold of a complex torus $T$. We assume that 
$X$ is the Zariski closure of an admissible Riemann surface ${\mathcal Y}$ for which $(\dagger)$ holds. Then $X$ is the translate of a sub-torus of $T$, or there is a variety
of general type $W$ and map $\cR: X\to W$ such that $\cR\circ\varphi$ has finite area.
\end{theorem}




\proof First we assume that $\cY$ verifies the hypothesis (a) in the definition 
Definition \ref{admiss}. We will follow the approach in \cite{Bloch}, \cite{Siu1}; starting with $V:= T_{X}$, we consider the tower of 
directed manifolds $(T_k, V_{k})_{k\geq 0}$, whose construction was recalled in the first part of this article.
Since $T$ is flat, we have
$$T_k= T\times \cR_{n, k}$$
where we recall that here $\cR_{n, k}$ is the ``universal'' rational homogeneous
variety $\displaystyle \bC^{nk}\setminus \{0\}/\bG_k$, cf. \cite{Dem2}. The curve 
$f: \cY\to T$ lifts to $T_k$, as already explained. We denote by
$$f_k: \cY\to T_k$$ 
the lift of $f$. 

Let $X_k$ be the Zariski closure of the image of $f_k$, and let
$$\tau_k: X_k\to \cR_{n, k}$$
be the composition of the injection $X_k\mapsto T_k$ with the projection on the second factor 
$T_k\to \cR_{n,k}$.
\smallskip

\noindent Then either the dimension of the generic fiber of $\tau_k$ 
is strictly positive 
for all $k$, or there exists a value of $k$ for which the map $\tau_k$ is generically finite. The geometric counterpart of each of these eventualities is 
analyzed along the two following statements, due to A. Bloch. After this is done, we will show that 
the proof of \ref{t11} follows almost immediately.  

\begin{proposition}\label{p4} We assume that for each $k\geq 1$
the fibers of $\tau_k$ are positive dimensional. Then the dimension of the 
subgroup $A_X$ of $T$ defined by
$$A_X:= \{a\in T: a+ X= X\}$$
is strictly positive.

\end{proposition}

\proof We fix a point $x_0\in \cY$ such that $f_k(x_0)$ is a regular point of $X_k$, for each $k\geq 1$. 
It is clear that such a point exists, since the image of the curve $f_k$ is dense in $X_k$, hence the inverse 
image of the singular set of $X_k$ by $f_k$ is at most countable. By the same argument, we can assume that the fiber of the map $\tau_k$ through $f_k(x_0)$ is ``generic'', meaning that it is positive dimensional.

By this choice of the point $x_0$ the fiber 
$$F_{k, 0}:= \tau_k^{-1}\tau_k \big(f_k(x_0)\big)$$
is positive dimensional, for each $k\geq 1$. Thus, there exists a curve $\gamma_k: (\bC, 0)\to X_k$ such that 
$$\gamma_k(t)= \big(z(t), \lambda\big)$$
such that $f_k(x_0)$ corresponds to the couple $\big(z(0), \lambda\big)$ according to the decomposition of $T_k$ and $\lambda$ is fixed (independently of $t$).
In particular, this shows that the dimension of the 
analytic set 
$$\Xi_k:= \{a\in T: f_k(x_0)\in J^k(X) \cap J^k(a+ X)\}$$
is strictly positive. This is so because for each $t$ close enough to zero 
we can define an element $a_t$ by the equality 
$$f_k(x_0)= a_t+ z(t)$$
and thus $a_t\in \Xi_k$ since the curve $t\to z(t)$ lies on $X$.
\smallskip

Next we see that we have the 
sequence of inclusions 
$$\Xi_1\supset \Xi_{2}\supset ... \supset \Xi_{k}\supset \Xi_{k+1}\supset ...$$
therefore by Noetherian induction there exists a large enough positive integer 
$k_0$ such that 
$$\Xi_k= \Xi_{k+1}$$
for each $k\geq k_0$, and such that the dimension of $\Xi_k$ is strictly positive. 
But this means that for every $\displaystyle a\in \Xi_{k_0}$, the 
image of the curve $f$ belongs to the translation 
$a+ X$ of the set $X$ (because this curve is tangent to an infinite order to the 
$a$-translation of $X$). Thus, given that the curve $f$ is Zariski dense, we will have $X= X+a$ .
The proposition is proved. \qed

\medskip

\noindent In the following statement we discuss the other possibility.

\begin{proposition}\label{p5}
Let $k$ be a positive integer such that the map
$$\tau_k: X_k\to \cR_{n, k}$$
has finite generic fibers. Then there exists a jet differential 
${\mathcal P}$ of order $k$ with values in the dual of an ample line bundle, and 
whose restriction to $X_k$ is non-identically zero.
\end{proposition}

\proof 

\noindent The proof relies on the following claim: {\sl the restriction to $X_k$ 
of the tautological bundle ${\mathcal O}_k(1)$ associated to $T_k$ is big}.
This condition is equivalent to the fact that for $m\gg 0$ large enough we have
$$\displaystyle H^0\big(X_k, {\mathcal O}_k(m)\otimes A^{-1}\big)\neq 0.$$
Indeed, in the first place we know that $\displaystyle {\mathcal O}_k(1)= \tau_k^\star \cO(1)$
where $\cO(1)$ is the tautological bundle on $\cR_{n, k}$. It also turns out that 
$\cO(1)$ is big (cf. \cite{Dem2}). Since the generic fibers of $\tau_k$ are of dimension zero, the 
inverse image of $\cO(1)$ is big. Indeed, the growth of the space of sections of $\tau_k^\star \cO(m)$
as $m\to \infty$
obtained by pull-back
proves it. The claim is therefore established.\qed

\bigskip

\noindent We show next that Theorem \ref{t11} follows from the two statements above. Let 
$$\varphi: \cY\to T$$ be a non-constant 
holomorphic map from a parabolic surface $\cY$ of finite mean Euler characteristic into a complex torus $T$. We denote by $X$
the Zariski closure of its image. Thanks to the result of Ueno we can consider the reduction map
$\cR:X\to W$ associated to $X$. We claim that under the hypothesis of Theorem \ref{t11} the area of the map $\cR\circ \varphi$ is finite. 

If $W$ is a point, this means that $X$ is the translate of a sub-torus. If this is not the case, then we can assume that $X$ is of general type and the area is infinite.
By the vanishing theorem \ref{t7} we see that the 
hypothesis of Proposition \ref{p5} will never be verified, for any $k\geq 1$ (otherwise $X_k$ would not be the Zariski closure of the lift of the curve). 
Hence the hypothesis of the Proposition \ref{p4}
are verified, and so $X$ will be invariant by a positive dimensional sub-torus of $T$. 
Since $X$ is assumed to be a manifold of general type, its automorphisms group is finite, so this cannot happen. In conclusion the area of
$\cR\circ \varphi$ must be finite and our result is proved. Observe here that the area involved is counted with respect to the parabolic exhaustion.
\medskip

\noindent If $f:\bD\to T$ is an admissible Riemann surface as in (b) of Definition \ref{admiss},
then the argument is basically unchanged, given the Remark \ref{rem}. We leave the details to the interested reader.
\medskip

\noindent If $f:\bC\to T$ is a non-constant holomorphic map of finite area, then $f$ extends to a map $\bP^1\to T$. But any such map must be constant, so we obtain a contradiction. In conclusion, 
Theorem \ref{t11} represents a generalization of classical Bloch theorem.

\begin{rem} {\rm J. Noguchi exhibited a simple example showing that the second possibility in Theorem \ref{t11} could actually occur. This helped us to correct a slip in a previous version of the current article.}
\end{rem}

\medskip

\subsection{Bloch-Ax-Lindemann theorem} As a consequence of the results obtained in the previous 
section, we establish next a refinement of the Ax-Lindemann theorem, cf.
\cite{PZ}, \cite{Ull}, \cite{KUY}. We assume that we have a co-compact lattice $\Lambda\subset \bC^k$ such that the quotient $T: =\bC^k/\Lambda$ is an abelian variety. Let $\pi: \bC^k\to T$ be the projection map.
\smallskip

\noindent Let $V\subset \bC^k$ be a connected analytic subset of pure dimension $p$. We define the following set $\cA(V)\subset V\times V$ consisting of couples of points
$(x, y)\in V\times V$ such that there exists an admissible curve
${\mathcal C}_{xy}\subset V$ with finite Euler characteristic and infinite area
passing through $x$ and $y$, such that the Zariski closure of $\pi(V)$ is a translate of a torus. The closed sets in the Zariski topology we consider in $\bC^k$ are the closed analytic subsets. 
\smallskip

\noindent Then we have the following statement. 

\begin{theorem}\label{bal} Let $V\subset \bC^k$ be a connected analytic set of pure dimension $p$. Then the following assertions hold true.
  \begin{enumerate}
    \smallskip

\item[(1)] If $p=1$ and $V$ is admissible with finite Euler characteristic, then the
  Zariski closure of $\pi (V)$ is the translate of a sub-torus, unless $\cR\circ \pi(V)$ is an algebraic curve of genus greater than one.    
 \smallskip

\item[(2)] If $p\geq 1$,
we assume that the closure $\overline{\cA(V)}^{\rm an}$ of the set $\cA(V)$ with respect to the analytic topology is equal to $V\times V$.
  Then the Zariski closure $\overline{\pi(V)}\subset T$ 
is the translate of a sub-torus.
\end{enumerate}
\end{theorem}

\begin{proof}
  We first consider the case where $V$ is a curve, i.e. $p=1$. 
  Given that $V$ is admissible,
  Theorem \ref{t11} shows that the Zariski closure $X$ of $\pi(V)$ is either the translate of a sub-torus or there exists a map
  $$\cR: X\to W$$ onto a general type subvariety $W$ of an abelian variety $A$
  such that the area of $\cR\circ \pi$ is finite. Assume that the second possibility occur. Then we claim that the image $W$ is a curve. Indeed,
  there exists a linear map $L: \bC^k\to \bC^l$ which induces the morphism
$T\to A$; let $\pi_1: \bC^l\to A$ be the covering map of $A$.   
  The area of $\pi_1\big(L(V)\big)\subset A$ is finite, since $L(V)\subset \bC^l$ is an analytic subset, and $\pi_1$ is a covering map. Therefore $W$ is one-dimensional, and its genus is greater than one because it is of general type. The point (1) is established. 
\smallskip

\noindent Next we consider the case $p> 1$. For every admissible curve ${\mathcal C}_{xy}$ as in the definition of $\cA(V)$, the Zariski closure of
$\pi({\mathcal C}_{xy})$ is the translate of a sub-torus, by definition.
In other words, there exists a linear subspace ${\mathcal L}_{xy}\subset \bC^k$ and a vector $v_{xy}$ such that
$$\overline{\pi\left({\mathcal C}_{xy}\right)}= \pi(v_{xy})+ \pi\left({\mathcal L}_{xy}\right)$$
In particular this means that for any $\lambda\in \bC$ we have
$$\displaystyle \lambda \pi(x)+ (1-\lambda)\pi(y)\in \overline{\pi(V)}.$$

\noindent For each complex parameter $\lambda$ we define the map
$$\Pi_\lambda: \bC^k\times \bC^k\to T, \quad \Pi_\lambda(x, y):= \lambda \pi(x)+ (1-\lambda)\pi(y).$$ 
Our previous considerations show that we
have $\displaystyle \Pi_\lambda\left(\cA(V)\right)\subset \overline{\pi(V)}$. As a consequence, we infer that
$$\displaystyle \Pi_\lambda\left(\overline{\cA(V)}^{\rm an}\right)\subset \overline{\pi(V)}$$
because $\Pi_\lambda$ is continuous with respect to the analytic Zariski topology. Now the hypothesis of \ref{bal} shows that
$$\displaystyle \Pi_\lambda\left(V\times V\right)\subset \overline{\pi(V)}$$
and this implies that we have
$$(1-\lambda)\overline{\pi(V)}+ \lambda\overline{\pi(V)}\subset \overline{\pi(V)}$$
for any complex value $\lambda\in \bC$. 

We deduce that given any two points $z_1$ and $z_2$ in the Zariski closure $\overline{\pi(V)}$ of the set $V$
there exists a holomorphic map $\varphi:\bC\to \overline{\pi(V)}$ whose image contains $z_j$. It follows that 
$\overline{\pi(V)}$ is the translate of a sub-torus, as we see by applying the Theorem of Ueno quoted above combined with a result due to Y. Kawamata (Theorem 4 in \cite{Kawa}). If $\overline{\pi(V)}$ is not the translate of a sub-torus, then it maps surjectively onto a general type algebraic sub-variety $W$ of an abelian variety $A$. Then Kawamata's result states that the 
set of all translates of abelian subvarieties in $W$ is a proper algebraic subset $Y\subset W$. We take two points $z_1$ and $z_2$ which do not project into $Y$ and which do not belong to the same fiber of 
the map $\cR: \overline{\pi(V)}\to W$. We conclude by Bloch theorem: the Zariski closure of the image of the map $\varphi$ as above in $W$ gives the translation of a sub-torus of $A$ which is contained in 
$W$ but not in $Y$ (hence we obtain a contradiction).
\end{proof}

\medskip

\noindent If $V$ is an algebraic subvariety of $\bC^n$, we obtain a more precise statement, cf. \cite{PZ}, \cite{Ull}, \cite{KUY}.

\begin{theorem}\label{vax} Let $V\subset \bC^k$ be an algebraic variety. Then the Zariski closure of the image 
$\pi(V)$
  is the translate of a sub-torus.
\end{theorem}

\begin{proof} We first assume that the dimension of $V$ is 1. The fact that $V$ is algebraic implies the existence of a linear projection $p:\bC^n\to \bC$
  such that the restriction map $p|_V: V\to \bC$ is proper and ramified in at most a 
  finite number of points. The map $p|_V$ induces a very explicit parabolic structure on $V$ (by pull-back of the usual one on the complex plane).
  With respect to this particular structure we infer that we have
$$\frac{T_f(r)}{\log r}\to \infty \leqno{(\star)}$$
as $r\to \infty$ where $f:V\to T$ is the restriction of the map $\pi: \bC^k\to T$ to $V$. Since  
the Euler characteristic of $V$ is finite, we can apply Theorem \ref{t11}. If the Zariski closure of the image of $V$ is translate of a sub-torus, we are done.

If not, there exists a map $q: T\to A$ from $T$ onto an abelian variety, such that the induced map $\cR: X\to W$ has the property that $\cR\circ f(V)$ has finite area. Then we argue as follows. The map $q$ is induced by a linear application $L: \bC^k\to \bC^l$. If $L(V)$ is constant, then it is trivial to conclude. Assume that this is not the case, so $L(V)$ is an algebraic curve in $\bC^l$
and $(\star)$ does apply; the contradiction obtained ends the proof of Theorem \ref{vax} if $\dim(V)=1$.
\smallskip

\noindent Let $V\subset \bC^k$ be an algebraic subset of arbitrary dimension. In this case we can apply point 2 in Theorem \ref{bal}.\end{proof}

\begin{ex}{\rm \hfill
    \smallskip
    
\noindent $\bullet$ Let $V\subset \bC^k$ be the hypersurface given by the equation
$$w^d+\sum_{j=1}^da_j(z)w^{d-j}= 0$$
where $a_j\in {\mathcal O}(\bC^{k-1})$.  If $k= 2$, then $V$ is a parabolic curve, since it is a finite (proper, ramified) cover of $\bC$.
The argument is immediate: given a bounded subharmonic function on $V$, its direct image (by the projection to $\bC$) must be constant. This implies that the function itself is constant.

The general case follows: let $x= (z_1, w_1)$ and $y= (z_2, w_2)$ be any two points of $V$. We consider the curve ${\mathcal C}_{xy}:= p^{-1}(\gamma_{z_1z_2})\cap V$, where $p: \bC^k\to \bC^{k-1}$ is the projection $(z, w)\to z$ and $\gamma_{z_1z_2}$ is the line determined by $(z_i)_{i=1,2}$ if $z_1\neq z_2$, and if $z_1= z_2$ 
we can take any line passing through the point. Again, the curve ${\mathcal C}_{xy}$ is parabolic.

 We now address the question of finiteness of the Euler characteristic. The function $ u(z):= \log^+|z| $is proper and harmonic out of a compact set of $V$. The critical points out of a compact are the points 
 where the 
 tangent to $V$ is vertical. If this number is finite, then so it is
 the Euler characteristic. By a careful choice of the
 coefficients $a_j$, this can be achieved.
 We see that the hypotheses are satisfied far beyond the algebraic case.

Observe that we have to check that the curves ${\mathcal C}_{xy}$
satisfy the previous conditions on growth of Euler characteristic, only for a Zariski dense set of $(x,y)$. 
 Then $V$ satisfies equally the hypothesis of Theorem \ref{bal}.
 
 It  is easy to construct a transcendental subvariety V , as above containing a Zariski dense sequence of Algebraic curves. Hence it will satify the hypothesis of Theorem 5.6. Indeed consider
 for simplicity
 $$w^d+ a(z)= 0,$$ with $a$ transcendental but which is a polynomial when restricted to
a countable Zariski dense set of lines through the origin. Then the above set $V$ satisfies the assumptions in Theorem 5.6.
  
}\end{ex}


\section{Parabolic Curves Tangent to Holomorphic Foliations}

\subsection{Nevanlinna's Currents Associated to a Parabolic Riemann Surface}

\noindent Let $(\cY, \sigma)$ be a parabolic Riemann surface. 
We fix a K\"ahler metric $\omega$ on $X$, and let $\varphi: \cY\to X$ be holomorphic map. 
For an open set $S\subset \cY$ with smooth boundary we denote by
$$\Vert \varphi (S)\Vert:= \int_{S}\varphi^\star \omega \leqno(25)$$
We define the (normalized) integration current 
$$T_S:= \frac{[\varphi_\star (S)]}{\Vert \varphi (S)\Vert}$$
which has bidimension (1,1) and total mass equal to 1. In general, the  current $T_S$ is positive but not closed. However, it may happen that for some accumulation point
$$T_\infty= \lim_kT_{S_k}$$
is a closed current, when $S_k\to \cY$.

\noindent 

\begin{theorem} \label{t12}
\cite{Burns} Let $(\cY, \sigma)$ be a parabolic Riemann surface; we denote by $S_j:= \bB(r_j)$ the parabolic balls of radius $r_j$, where 
$(r_j)_{j\geq 1}$ is a sequence of real numbers such that $r_j\to \infty$.
We consider a holomorphic map $\varphi: \cY\to X$ of infinite area. Then there
exists at least one accumulation point of the sequence of currents 
$$T_j:= \frac{[\varphi_\star (S_j)]}{\Vert \varphi (S_j)\Vert}$$ 
which is a closed (positive) current, denoted by $T_{\infty}$. 
\end{theorem}

\begin{proof} The arguments presented here are a quantitative version of the ones 
in \cite{Burns}. 

We denote by $u:= \log \sigma$ the log of the exhaustion function; by hypothesis, the measure $dd^cu$ has compact support. We define a function $H$ on $\cY$ by the equality
$$\varphi^\star \omega:= H du\wedge d^cu.$$
Let $A(t):= \int_{(u<t)}\varphi^\star\omega$ be the area of the image of parabolic ball of radius $e^t$ with respect to the
metric $\omega$, and let 
$$L(t):= \int_{(u=t)}\sqrt {H}d^cu;$$
geometrically, it represents the length of the parabolic sphere of radius $e^t$ measured with respect to the metric induced by $Hdu\wedge d^cu$ (or the length of the image with respect to $\omega$).

By Cauchy-Schwarz inequality we have 
$$L(t)^2\leq \int_{(u=t)}{H}d^cu \int_{\bB(t)}dd^c u= \int_{(u=t)}{H}d^cu \leqno(26)$$
because $\int_{\bB(t)}dd^c u= 1$. On the other hand we have 
$$\frac{d}{dt}A(t)= \int_{(u=t)} {H}d^cu\leqno(27)$$
thus combining the inequalities (26) and (27) we obtain
$$L(t)^2\leq \frac{d}{dt}A(t).$$
For every positive $\varepsilon$ we have $\displaystyle \frac{d}{dt}A(t)\leq A^{1+ 2\varepsilon}(t)$ for any $t$ belonging to the complement of a set $\Lambda_\varepsilon$
of finite measure; as a result we infer that the inequality
$$L(t)\leq A(t)^{1/2+ \varepsilon}\leqno (28)$$
holds true for any $t\in \bR_+\setminus \Lambda_\varepsilon$. In particular, this implies the existence of a current 
as an accumulation point of $T_j$, and the Theorem \ref{t12} is proved.\end{proof}
\smallskip

\noindent We consider next the case where $dd^cu$ is not 
necessarily with compact support.
We see that the previous statement admits the following version.
Let
$$\rho(t):= \int_{\bB(t)}dd^c u$$
be the mass of the measure $dd^cu$ on the ball of radius $t$. The 
inequalities (26) and (27) above show that we have
$$L(t)^2\leq \rho(t)\frac{d}{dt}A(t).$$
As already seen, we have $\displaystyle \frac{d}{dt}A(t)\leq A^{1+ 2\varepsilon}$ for any $t$ in the complement
of a set $\Lambda_\varepsilon$ of finite measure. Thus we have
$$L(t)\leq \left(\rho(t)\right)^{1/2}A(t)^{1/2+ \varepsilon},$$
and we see that we get a closed current as soon as there exists a constant $c> 0$ such that we have
$$\rho(t)\leq cA(t)^{1-\varepsilon'}$$
for some positive $\varepsilon'$.

\medskip

\noindent By using similar arguments, combined with results by B. Kleiner \cite{Kleiner} and
B. Saleur \cite{Saleur} we obtain a result in the direction of the conjecture in paragraph 1. Let $E\subset \bD$ be a polar subset of the unit disk, and let 
$f: \bD\setminus E\to M$ be a holomorphic map with values in a compact, Kobayashi hyperbolic manifold $M$. 
As it is well-known (cf \cite{Tsuji}) the set $\bD\setminus E$ carries a local exhaustion function $\sigma$ such that 
$u:= \log \sigma$ is harmonic. Let $\chi(t)$ be the Euler characteristic of the domain $(\sigma< t)$. We define
$A(t):= \int_{\sigma< t}f^\star\omega$. 

\begin{corollary}\label{c2}
If we have $\displaystyle \frac{|\chi(t)|}{A(t)}\to 0$ as $t\to \infty$, then the map $f$ admits an extension through $E$.
\end{corollary}

\begin{proof}
The argument relies heavily on the following result, which uses a technique due to B. Kleiner \cite{Kleiner}.

\begin{theorem} \label{t13}
\cite{Saleur} Let $(M, \omega)$ be a compact Kobayashi hyperbolic manifold. There exist two constants $C_1, C_2$ such that for every holomorphic map $f: \Sigma\to M$ defined on a Riemann surface with smooth boundary $\partial \Sigma$ we have
$${\rm Area}\big(f(\Sigma)\big)\leq C_1|\chi(\Sigma)|+ C_2{\rm L}\big(f(\partial \Sigma)\big).\leqno(29)$$
\end{theorem}
Coming back to the domains $(\sigma< t)$, inequality (28) shows that we have
$$L(t)\leq A(t)^{1/2+ \varepsilon}.$$
When combined with the inequality (29) of the preceding theorem, we get
$$A(t)\leq C_1|\chi(t)|+ C_2 A(t)^{1/2+ \varepsilon}.$$
Given the hypothesis concerning the growth of the Euler characteristic, we infer the existence of a constant $C_3$ such that
$$A(t)\leq C_3$$
and for the rest of the proof we will follow the argument given in Proposition 1.1. Indeed, the current associated to the 
graph $\Gamma_f\subset \bD\setminus E
\times M$ of the map $f$ has finite mass near the polar set $E\times M\subset \bD\times M$.
By using Skoda-ElMir extension theorem (for a simple proof, see \cite{Sibony}), the current $[\Gamma]$ extends to $\bD\times M$ with no mass 
on $E\times M$. But this implies that 
the graph $\Gamma$ extends as an analytic subset of $\bD\times M$; in other words, $f$ extends as a meromorphic map.
It follows that in fact $f$ is holomorphic, since $f$ is defined on a 1-dimensional disk.
\end{proof}
\smallskip

\begin{rem} {\rm
We can also consider a version of the current $T_\infty$ in the above statement. For each $r> 0$, the expression
$$T_r:= \frac{1}{T_f(r)}\int_0^r\frac{dt}{t}[\varphi_\star(\bB_t)]$$
defines a positive current on $X$. One can show that there exists a sequence
$r_k$ such that the limit points of $\displaystyle (T_{r_k})$ are positive and closed.
Any such limit will be called a Nevanlinna current associated to $\varphi$, and will be denoted by
$T[f]$. If we consider the lift of $f$ to $\bP(T_X)$, we get (with the same construction) a current
denoted $T[f^\prime]$. Let $\pi: \bP(T_X)\to X$ be the projection; then we can assume 
that $\pi_\star(T[f^\prime])= T[f]$, as we will see later.
}
\end{rem}

\noindent The preceding considerations apply e.g. to maps $f: \bD\to X$ defined on the unit disk $\bD\subset \bC$; in this case we have
$\displaystyle u:= \log\frac{1}{1- |z|}$, so $f$ will define a closed positive current provided that its area grows fast enough; this is the content of the next statement.

\begin{corollary}
Let $f: \bD\to X$ be a holomorphic map. We denote by $T(r)$ the Nevanlinna characteristic of $f$, and we assume that we have 
$$\frac{T(r)}{\log\frac{1}{1- r}}\to \infty$$
as $r\to 1$. Then any limit of $T_r(f)$ is a closed positive current. 
\end{corollary}

\medskip

\subsection{Metrics on the tangent bundle of a holomorphic foliation by disks}

Let $\cF$ be a 1-dimensional holomorphic foliation (possibly with singularities) on a manifold $X$. 
This means that we are given a finite open 
covering $(U_\alpha)_{\alpha}$ of $X$ with coordinates charts, and a family of associated vector fields 
$\displaystyle v_\alpha\in H^0(U_\alpha, T_X|_{U_\alpha})$
such that there exists $g_{\alpha \beta}\in \cO^\star (U_\alpha\cap U_\beta)$ 
with the property that 
$$v_\alpha= g_{\alpha \beta}d\pi_{\alpha \beta}(v_\beta)$$
on the intersection of $U_\alpha$ and $U_\beta$. Here we denote by $(\pi_{\alpha \beta})$ the transition functions of $X$, 
corresponding to the covering $(U_\alpha)$.
The (analytic) set of zeros of $(v_\alpha)$ is supposed to have codimension at least two, and it is
denoted by $\cF_{\rm sing}$. The functions $\displaystyle (g_{\alpha \beta})$ verify the cocycle property, and they define the cotangent bundle corresponding to the foliation $\cF$, denoted by 
$T_{\cF}^\star$.

From the global point of view, the family of vector fields $\displaystyle (v_\alpha)_\alpha$
corresponds to a section $V$ of the vector bundle $T_X\otimes T_{\cF}^\star$.
\smallskip

\noindent Let $\omega$ be a metric on $X$ (which is allowed to be singular). We will show that 
$\omega$ induces a metric 
$\displaystyle h_s$ on the tangent bundle $T_{\cF}$ (as we will see, the induced metric may be singular even if the reference metric $\omega$ is smooth).

Let $x\in X$, and let $\xi\in T_{\cF, x}$ be an element of the fiber at $x$ of the tangent bundle corresponding to $\cF$. Then we define its norm as follows
$$\vert \xi\vert_{h_s}^2:= \vert V_{x}(\xi)\vert_{\omega}^2.\leqno (30)$$
The local weights of the metric $\displaystyle h_s$ on the set $U_\alpha$ are described as follows.
Let $z^1,\dots z^n$ be local coordinates on $X$ centered at $x$. We write
$$v_\alpha = \sum_{i=1}^n a^i_\alpha\frac{\partial}{\partial z^i}$$
where $a^i_\alpha$ are holomorphic functions defined on $U_\alpha$; we assume that their 
common zero set has codimension at least 2 in $X$. 

The local weight $\phi_\alpha$ of the metric $\displaystyle h_s$ is given by the expression
$$\phi_\alpha= -\log \sum_{i, j}a^i_\alpha\overline{a^j_\alpha}\omega_{i\overline j}$$
where $\omega_{i\overline j}$ are the coefficients of the metric $\omega$ with respect to the local coordinates $(z^j)_{j=1,\dots n}$. Indeed, let $\theta$ be a local trivialization of the bundle $T_{\cF}$.
Then according to the formula (30) we have
$$\vert \xi\vert_{h_s}^2= 
\big(\sum_{i, j}a^i_\alpha\overline{a^j_\alpha}\omega_{i\overline j}\big)|\theta(\xi)|^2,\leqno (31)$$
which clarifies the formula for the local weight of $h_s$.

\medskip

In some cases, the previous construction can be further refined, as follows.
\smallskip

\noindent Let $B= \sum_{j=1}^NW_j$ be a divisor on $X$. We assume that the following requirements are fulfilled.

\begin{enumerate}

\item[(a)] At each point of 
$x\in \Supp (B)$ the local equations of the analytic sets 
$$\displaystyle (W_j, x)_{j=1,\dots,k}$$ 
can be completed to a local coordinate system centered at $x$. Here we denote by $k$ the number of hypersurfaces in the
support of $B$ containing the point $x$ (and we make a slight abuse of notation). In the language of algebraic geometry, one calls such a pair $(X, B)$ \emph {log-smooth}.
\smallskip

\item[(b)] We assume that each component $W_j$ of $\Supp (B)$ is invariant by the foliation $\cF$.

\end{enumerate}

\noindent If the condition (a) above is verified, then we recall that the 
\emph{logarithmic tangent bundle} of ($X, B$) is the subsheaf of $\cO(T_X)$ defined locally as follows.

Let $U\subset X$ be a coordinate open set. We assume that we have a coordinate system
$z_1,\dots, z_n$ on $U$, such that 
$$\Supp(B)\cap U= \big(z_1z_2\dots z_k= 0\big).$$
Then the logarithmic tangent bundle $T_X\langle B\rangle $ corresponding to the pair $(X, B)$ is the subsheaf of $T_X$ 
whose local sections on $U$ are given by
$$v= \sum_{j=1}^kv_jz_j\frac{\partial}{\partial z_j}+ \sum_{p=k+1}^nv_p\frac{\partial}{\partial z_p}.$$
In other words, the local sections of $T_X\langle B\rangle|_U $ are the vector fields of $T_X|_U$
which are tangent to $B$ when restricted to $B$.
We note that the $T_X\langle B\rangle$ is a vector bundle of rank $n$, and the local model of a hermitian metric on it is given by
$$\omega_{U}\equiv \sqrt{-1}\sum_{j=1}^k\frac{dz_j\wedge d\ol z_j}{|z_j|^2}+
 \sum_{j=k+1}^n dz_j\wedge d\ol z_j$$
 i.e. a metric with logarithmic poles along $B$. So, we have 
 $$|v|_{\omega_U}^2= \sum_j|v_j|^2.$$
From a global point of view, a hermitian metric $\omega_{X, B}$ on $T_X\langle B\rangle$ can be written as
\begin{equation} 
\begin{split}
\omega_{X, B}|_U= & \sqrt{-1}\sum_{j, i=1}^k\omega_{j\ol i}\frac{dz_j\wedge d\ol z_i}{z_j\ol z_i}+
2{\rm Re}\sqrt{-1}\sum_{j>k\geq  i}\omega_{j\ol i}\frac{dz_j\wedge d\ol z_i}{\ol z_i}+\\
+ & \sqrt{-1}\sum_{j, i\geq k+1}\omega_{j\ol i}{dz_j\wedge d\ol z_i}\\
\nonumber
\end{split}
\end{equation}

where the Hermitian matrix
$(\omega_{j\ol i})$ is positive definite.
\smallskip

\noindent If moreover the condition (b) is fulfilled, then the family of vector fields $v_\alpha$ defining the foliation $\cF$ can be seen as a global section $V_B$ of the bundle 
$$T_{X}\langle B\rangle\otimes T_{\cF}^\star$$
and we have the following version of the metric constructed above. For each vector 
$\xi\in T_{\cF, x}$ we define
$$\Vert \xi\Vert^2_{h_{s, B}}:= |V_{B, x}(\xi)|^2_{\omega_{X, B}}.$$
As in the case discussed before, we can give the local expression of the 
metric on $T_{\cF}$, as follows. Let
$$v_\alpha = \sum_{i=1}^k a^i_\alpha z_i\frac{\partial}{\partial z^i}+ 
\sum_{i=k+1}^n a^i_\alpha \frac{\partial}{\partial z^i}
$$
be a logarithmic vector field trivializing the tangent bundle of the foliation on a coordinate set
$U_\alpha$. Then 
the local weight $\phi_{\alpha, B}$ of the metric $\displaystyle h_{s, B}$ induced by the 
metric $\omega_{X, B}$ is given by the expression
$$\phi_{\alpha, B}= -\log \sum_{i, j}a^i_\alpha\overline{a^j_\alpha}\omega_{i\overline j}.
\leqno(32)$$
In particular we see that this weight is \emph{less singular than} the one in the expression (31). This will be crucial in the applications.
\smallskip

As far as the curvature current is concerned, the metric $\displaystyle h_s$ as well as its logarithmic variant $\displaystyle h_{s, B}$ seem useless: given the definition above, its associated curvature is neither positive nor negative.
Indeed, $\phi_\alpha$ may tend to infinity along the singular set of the foliation $\cF$, and 
it may tend to minus infinity along the singularities of the metric $\omega$. However, we will present a few
applications of this construction in the next paragraphs.

\subsection{Degree of currents associated to parabolic Riemann surfaces on the tangent bundle of foliations}

Let $(X, \omega)$ be a compact complex hermitian manifold, and let $\cF$ be a holomorphic foliation on
$X$ of dimension 1. 
Let $f:\cY\to X$ be a holomorphic map, where $(\cY, \sigma)$ is 
a parabolic Riemann surface tangent to $\cF$,
and let 
$$T[f]:= \lim_rT_r[f]$$
be a Nevanlinna current associated to it. 

\noindent In this section we will derive a lower bound in arbitrary dimension for the quantity 
$$\int_XT[f]\wedge c_1(T_{\cF}),$$
in the same spirit as \cite{McQ}, \cite{Brun1}. Prior to this, we introduce a few useful notations.

Let $\cJ_{\cF_{\rm s}}$ be the coherent ideal associated to the singularities of $\cF$; this means that 
locally on $U_\alpha$ the generators of $\cJ_{\cF_{\rm s}}$ are precisely the coefficients 
$(a_\alpha)$ of the vector $v_\alpha$ defining $\cF$, i.e. 
$$v_\alpha= \sum_{i=1}^n a_\alpha^i\frac{\partial}{\partial z^i}.$$ 

As we have already mentioned in paragraph 3,
there exists a function $\psi_{\rm sing}$ defined on $X$ and 
having the property that 
locally on each open set $U_\alpha$ we have 
$$\psi_{\rm sing}\equiv \log |v_\alpha|_\omega^2$$ modulo a bounded function. 

Let $B= \sum_j W_j$ be a divisor on $X$, such that the pair $(X, B)$ satisfies the requirements (a) and
(b) in the preceding paragraph. Then the local generator of $T_\cF$ can be written in this case as 
$$v_{\alpha, B}= \sum_{i=1}^k a_\alpha^iz_i\frac{\partial}{\partial z^i}+ \sum_{i=k+1}^n a_\alpha^i\frac{\partial}{\partial z^i}.\leqno(33)$$
We denote by $\cJ_{\cF_{\rm s, B}}$ the coherent ideal defined by the functions $(a_\alpha^i)$
in (33). Then we have 
$$\cJ_{\cF_{\rm s}}\subset \cJ_{\cF_{\rm s, B}},$$ 
and the inclusion may be strict. We denote by $\psi_{\rm sing, B}$ the associated function.

The counting function with respect to the ideal defined by 
$\cF_{\rm sing}$ will be denoted
$$\displaystyle N_{f, \cJ_{\cF_{\rm s}}}(r)= \sum_{0<\sigma(t_j)< r}\nu_j\log \frac{r}{\sigma(t_j)}=
\int_0^r\frac{dt}{t}\int_{B(t)}\big(dd^c\psi_{\rm sing}\circ f\big)_{s},$$
with $f(t_{j})\in \Supp(\cF_{\rm sing})$, and the subscript $s$ above denotes the singular part of the 
considered measure. 
Its normalized expression will be written as
$$\nu^T(f, \cF_{\rm sing})(r):= \frac {1}{T_f(r)}N_{f, \cJ_{\cF_{\rm s}}}(r).\leqno (34)$$
The upper limit of the expression above will be denoted by
$$\nu^T(f, \cF_{\rm sing}):= \overline \lim_r\nu^T(f, \cF_{\rm sing})(r).$$
If $\Xi$ is an arbitrary analytic subset of $X$, we will denote by 
$\displaystyle \nu^T(f, \Xi)$ the quantity defined in a similar manner 
by using the function $\psi_\Xi$ instead of 
$\psi_{\rm sing}$.

We define the counting function with respect to $\cF_{\rm s, B}$ as 
$$\displaystyle N_{f, \cJ_{\cF_{\rm s, B}}}(r)= \sum_{0<\sigma(t_j)< r}\nu_j\log \frac{r}{\sigma(t_j)},$$
with $f(t_{j})\in \Supp(\cF_{\rm s, B})$,
together with it normalized expression
$$\nu^T(f, \cF_{\rm sing, B})(r):= \frac {1}{T_f(r)}N_{f, \cJ_{\cF_{\rm s, B}}}(r).\leqno (35)$$
The following truncated counting function will appear in our next computations:
$$\displaystyle N^{(1)}_{f, \cJ_{\cF_{\rm s}}\cap B}(r)= \sum_{0<\sigma(t_j)< r, f(t_j)\in B}\log \frac{r}{\sigma(t_j)}.\leqno (36)$$
and let $\displaystyle \nu^T_1(f, \cF_{\rm sing}\cap B)$ be its normalized upper limit.

We also recall the definition of 
$$m^T(f, \cF_{\rm sing}):= \overline \lim_r\frac{1}{T_f(r)}\int_{S(r)}-\psi_{\rm sing}\circ f d\mu_r$$
which is the (normalized) asymptotic 
proximity function for $f$ with respect to the ideal $\cJ_{\cF_{\rm s}}$, together with its logarithmic variant
$$m^T(f, \cF_{\rm sing, B}):= \overline \lim_r\frac{1}{T_f(r)}\int_{S(r)}-\psi_{\rm sing, B}\circ f d\mu_r.$$

The ramification function corresponding to $f$ is
$$R_{f}(r)= \sum_{0<\sigma(t_j^\prime)< r} 
{\mu_j}\log \frac{r}{\sigma(t_j^\prime)};\leqno(37)$$
so that $\mu_j$ is the vanishing order of $f^\prime$ at $t_j^\prime$. 
The curve $f$ is tangent to $\cF$, therefore for each open set $\Omega\subset \cY$ such that 
$f(\Omega)\subset U_\alpha$ for some index $\alpha$
we can write
$$f^\prime(t)= \lambda(t)v_{\alpha, f(t)}$$
for some function $\lambda$ which is holomorphic on $\Omega\setminus f^{-1}(\cF_{\rm sing})$.
We remark that 
if $f(t_j)\not\in \cF_{\rm sing}$, then the multiplicities $\mu_j$ above
coincide with the vanishing order of $\lambda$ evaluated at the critical points of $f$. 
It will be useful in what follows to have the decomposition 
$$R_{f}(r):= M_f(r)+ N_f({\rm Ram}, r)$$
according to the possibility that the critical value $f(t_j^\prime)$ of $f$ belongs to the set
$\cF_{\rm sing}$ or not.
We are using the notations
$$N_f({\rm Ram}, r):= \sum_{0<\sigma(t_j^\prime)< r, f(t_j^\prime)\not\in \cF_{\rm sing}} 
{\mu_j}\log \frac{r}{\sigma(t_j^\prime)},$$
and
$$M_f(r):= \sum_{0<\sigma(t_j^\prime)< r, f(t_j^\prime)\in \cF_{\rm sing}} 
{\mu_j}\log \frac{r}{\sigma(t_j^\prime)}.$$
Finally, the asymptotic normalized ramification of $f$ is denoted by
$${\overline \nu}({\rm Ram}, f):= \overline \lim_r \frac {1}{T_f(r)}R_{f}(r).$$

\medskip

\noindent We establish next the following general result, which gives an estimate
of the quantity $\int_XT[f]\wedge c_1(T_{\cF})$ in terms of the intersection of $f$ with the
singularities of the foliation. Our statement is a quantitative expression of the
fact that the derivative of $f$ can be seen as a meromorphic section of $f^\star T_{\cF}$.

\begin{theorem}
\label{monq}
Let $(X, \cF)$ be a compact complex manifold endowed with a holomorphic $1$-dimensional 
foliation $\cF$. Let $(\cY, \sigma)$ be a parabolic Riemann surface 
and let $f:\cY\to X$ be a holomorphic map whose image is tangent to $\cF$
such that $\mathfrak{X}_\sigma (r)= o(T_f(r))$. We assume that the image of $f$ is Zariski dense. Then we have
$$\int_XT[f]\wedge c_1(T_{\cF})\geq -\nu^T(f, \cF_{\rm sing})- m^T(f, \cF_{\rm sing})+
\overline \nu({\rm Ram}, f)
\leqno (\star)$$ 

\end{theorem}

\begin{proof} Let $\omega$ be a smooth metric on $X$, and let $h_s$ be the metric induced 
on $T_\cF$ by the procedure described in the preceding sub-section.

Let $r> t> 0$; we begin by evaluating the quantity 
$$\int_X T_r[f]\wedge \Theta_{h_s}(T_\cF) $$
and to this end we introduce the notations
$$\bB_{\alpha, \varepsilon}(t):= \{z\in \bB(t) : f(z)\in U_\alpha,\hbox{ and } { d_\omega }\big(f(z), \cF_{\rm sing})\geq \varepsilon\}$$
as well as its complement set inside the parabolic ball of radius $t$
$$\bB_{\alpha, \varepsilon}^c(t):= \bB_\alpha(t)\setminus \bB_{\alpha, \varepsilon}(t)$$
where $\bB_\alpha(t):= \bB(t)\cap f^{-1}(U_\alpha)$. 
Let $(\rho_\alpha)$ be a partition of unit corresponding to the cover $(U_\alpha)$. 

In the definition of the metric $h_s$ we use the smooth K\"ahler metric $\omega$ we have fixed on $X$, and we have.
\begin{equation} 
\begin{split}
\int_{X}T_r[f]\wedge \Theta_{h_s}(T_{\cF})= & 
-\sum_\alpha 
\frac {1}{T_f(r)}\int_1^r\frac{dt}{t}\int_{B_{\alpha,\varepsilon}^c (t)}\rho_\alpha(f)
f^\star dd^c
\log \vert v_\alpha\vert ^2_{\omega} \\
- & \sum_\alpha \frac {1}{T_f(r)}\int_1^r\frac{dt}{t}\int_{B_{\alpha,\varepsilon} (t)}\rho_\alpha(f)
f^\star dd^c
\log \vert v_\alpha\vert ^2_{\omega}.
\nonumber
\end{split}
\end{equation}
We remark that for each $t< r$ and for each index $\alpha$ we have
\begin{equation} 
\begin{split}
\int_{B_{\alpha,\varepsilon} (t)}\rho_\alpha(f)
f^\star dd^c
\log \vert v_\alpha\vert ^2_{\omega}
= & \int_{B_{\alpha,\varepsilon} (t)}\rho_\alpha(f)
dd^c
\log \vert f^\prime \vert ^2_{\omega}\\
- & \sum_{0<\sigma(t_j^\prime)< t, f(t_j^\prime)\not\in\cF_{\rm sing}} 
\rho_\alpha\big(f(t_j^\prime)\big){\mu_j}\delta_{t_j^\prime}.
\nonumber
\end{split}
\end{equation}
The equality in the formula above is due to the fact that locally at each point 
in the complement of 
the set $\cF_{\rm sing}$ we have
$$f^\prime(t) = \lambda v_{f(t)}$$
for some holomorphic function $\lambda$. We also remark that the relation above is valid for 
any $\varepsilon> 0$, and if we let $\varepsilon\to 0$, we have 
\begin{equation} 
\begin{split}
\lim_{\varepsilon\to 0}\int_1^r\frac{dt}{t}\int_{\bB_{\alpha, \varepsilon} (t)}
\rho_\alpha(f)dd^c\log \vert f^\prime\vert ^2_{\omega}= & 
\int_1^r\frac{dt}{t}\int_{\bB_\alpha(t)}\rho_\alpha(f)dd^c\log \vert f^\prime\vert ^2_{\omega} \\
- & \sum_{f(t_j^\prime)\in \cF_{\rm sing}} \rho_\alpha\big(f(t_j^\prime)\big)\mu_j\log \frac{r}{\sigma(t_j^\prime)}
\nonumber
\end{split}
\end{equation} 
as well as
$$
\lim_{\varepsilon\to 0}\int_1^r\frac{dt}{t}\int_{\bB_{\alpha, \varepsilon}^c(t)}
\rho_\alpha(f)f^\star dd^c
\log \vert v_\alpha\vert ^2_{\omega} = \sum_{t_j\in \bB_\alpha(t), f(t_{j})\in \cF_{\rm sing}} \rho_\alpha\big(f(t_j)\big)\nu_j\log \frac{r}{\sigma(t_j)}.$$
Therefore we obtain
\begin{equation} 
\begin{split}
\langle T_r[f], \Theta_{h_s}(T_\cF) \rangle \geq & -\nu^T(f, \cF_{\rm sing})(r)+ \frac {1}{T_f(r)}M_f(r) \\
- & \frac {1}{T_f(r)}\int_1^r\frac{dt}{t}\int_{\bB(t)}
dd^c\log \vert f^\prime \vert ^2_{\omega}+ \frac {1}{T_f(r)}N_f({\rm Ram}, r)= \\
= & - \nu^T(f, \cF_{\rm sing})(r)+ \frac {1}{T_f(r)}R_f(r)-\\
& \frac {1}{T_f(r)}\int_{S(r)}\log \vert f^\prime \vert ^2_{\omega}d\mu_r.
\nonumber
\end{split}
\end{equation}

\noindent Let $h:= h_s\exp(-\psi_{\rm sing})$; it is a metric with \emph{bounded} weights
of $T_\cF$, hence we can use it in order to compute the quantity we are interested in, namely
$$\int_XT[f]\wedge c_1(T_\cF)= \lim_r\langle T_r[f], \Theta_h(T_\cF) \rangle.$$
We recall that we have the formula 
$$\Theta_h(T_\cF)= \Theta_{h_s}(T_\cF)+ dd^c\psi_{\rm sing},$$
so as a consequence we infer that we have
$$\int_X T_r[f]\wedge \Theta_{h}(T_{\cF}) = \langle T_r[f]\wedge \Theta_{h_s}(T_{\cF})\rangle+ \frac{1}{T_f(r)}\int_{S(r)}\psi_{\rm sing}\circ f d\mu_r+ o(1).
$$
By combining the relations above we infer that
we have 
\begin{equation} 
\begin{split}
\int_X T_r[f]\wedge \Theta_{h}(T_{\cF})\geq & -\nu^T(f, \cF_{\rm sing})(r)+ \frac{1}{T_f(r)}\int_{S(r)}\psi_{\rm sing}\circ f d\mu_r \\
 & + \frac {1}{T_f(r)}R_{f}(r) -
 \frac {1}{T_f(r)}\int_{S(r)}\log \vert f^\prime \vert ^2_{\omega}d\mu_r+ o(1).
\nonumber
\end{split}
\end{equation}

\noindent By the logarithmic derivative lemma (or rather by an estimate as in Theorem 4.2) 
the last term of the preceding relation tends to a positive value, as
$r\to\infty$, hence we obtain
$$\int_X T_r[f]\wedge \Theta_{h}(T_{\cF})\geq -\nu^T(f, \cF_{\rm sing})- m^T(f, \cF_{\rm sing})+
\overline \nu({\rm Ram}, f)$$
and Theorem \ref{monq} is proved. 
\end{proof}
\medskip

\noindent Before stating a version of 
Theorem \ref{monq}, we note the following observations.
The lower bound obtained in Theorem \ref{monq} admits an easy interpretation, as follows.

Let $\cJ$ be the ideal sheaf defined by the scheme $\cF_{\rm sing}$; locally, this ideal is generated by the 
coefficients of the vectors $(v_\alpha)$ defining the foliation $\cF$. Let $p:\wh X\to X$ be 
a principalization of $\cJ$, so that $p^\star (\cJ)= \cO(-D)$ for some (normal crossing) effective divisor $D$ on
$\wh X$. According to Theorem 3.5, we have
$$T_{\wh f, \Theta_{D}}(r)= N_{\wh f, \cJ}(r)+ m_{\wh f, \cJ}(r)+ \cO(1)$$
where $\wh f$ is the lift of the map $f$ to $\wh X$. As a consequence, we infer the 
relation
$$\int_{\wh X}T[\wh f]\wedge c_1(D)\geq \nu^T(\wh f, \cF_{\rm sing})+ m^T(\wh f, \cF_{\rm sing})$$ 
and therefore Theorem \ref{monq} applied to $\wh f, \wh \cF$ can be restated as follows.
\medskip

\begin{corollary} We have the inequality
$$\displaystyle \int_{\wh X}T[\wh f]\wedge \big(c_1(T_{\wh\cF})+ c_1(D)\big)\geq 0.\leqno(38)$$ 
\end{corollary}
\medskip

\noindent In a similar framework, we note the following
``tautological" inequality in parabolic context.

\begin{lemma}\label{trivia}
Let $f:(\cY, \sigma) \to X$ be a holomorphic curve, where $\cY$ is a parabolic Riemann surface. 
We denote by $f_1: \cY\to \bP(T_X)$ the lift of $f$. Let $h$ be a hermitian metric on
$X$; we denote by the same letter the metric induced on the bundle $\cO(-1)\to \bP(T_X)$.
Then we have
$$\langle T_r[f_1], \Theta_h\big(\cO(-1)\big)\rangle \geq -C\big(\log T_f(r)+ \log r+ \mathfrak{X}_\sigma (r)\big)$$
where the positive constant $C$ above depends on $(X, h)$. In particular, 
if $\displaystyle \frac{\mathfrak{X}_\sigma (r)+ \log r}{T_f(r)}\to 0$ as $r\to \infty$, then we 
can 
construct a Nevanlinna current $T[f_1]$ associated to $f_1$ such that 
$\pi_\star T[f_1]= T[f]$, and we infer 
$$\langle T[f_1], \Theta_h\big(\cO(-1)\big)\rangle \geq 0.$$
\end{lemma}

\begin{proof}
We note that the derivative $f^\prime:= df(\xi)$ can be seen as a section of
the bundle $f_1^\star\cO(-1)$. So the lemma follows from Jensen formula combined with logarithmic derivative lemma
by an argument already used in Theorem 4.1 for sections of $\displaystyle \cO_{X_k}(m)\otimes A^{-1}$; we offer no further details.
\end{proof}

\noindent We turn next to the logarithmic version of Theorem \ref{monq}.  
 
\begin{theorem}
\label{monq, B}
Let $(X, B)$ be a log-smooth pair, such that every component of the support of $B$ is invariant by the foliation $\cF$. Let $(\cY, \sigma)$ be a parabolic Riemann surface 
and let $f:\cY\to X$ be a holomorphic map whose image is tangent to $\cF$
such that $\mathfrak{X}_\sigma (r)= o(T_f(r))$. We assume that the image of $f$
is not contained in the set $\Supp(B)$. Then we have
the inequality
$$\int_XT[f]\wedge c_1(T_{\cF})\geq -\nu^T(f, \cF_{\rm sing, B})- \nu^T_1(f, \cF_{\rm sing}\cap B)
-m^T(f, \cF_{\rm sing, B}).
\leqno (\star_{\rm B})$$ 
\end{theorem}
\medskip

\noindent The proof of Theorem \ref{monq, B} follows from the arguments we have used for 
\ref{monq}. The additional negative term in the statement 
is due to the singularities of the metric $\omega_{X, B}$:
\begin{equation} 
\begin{split}
\lim_{\varepsilon\to 0}\int_1^r\frac{dt}{t}\int_{\bB_{\alpha, \varepsilon}(t)}
\rho_\alpha(f)dd^c
\log \vert f^\prime\vert ^2_{\omega_{X, B}} = & \int_1^r\frac{dt}{t}\int_{\bB_{\alpha}(t)}
\rho_\alpha(f)dd^c
\log \vert f^\prime\vert ^2_{\omega_{X, B}}\\
- & 
\sum_{f(t_j^\prime)\in \cF_{\rm sing}\setminus B} \rho_\alpha\big(f(t_j^\prime)\big)\nu_j\log \frac{r}{\sigma(t_j^\prime)}\\
+ & \sum_{f(t_j^{\prime\prime})\in \cF_{\rm sing}\cap B} \rho_\alpha\big(f(t_j^{\prime\prime})\big)\log \frac{r}{\sigma(t_j^{\prime\prime})}.
\nonumber
\end{split}
\end{equation} 
where we denote by $t_j^\prime$ the critical points of $f$. The points $t_j^{\prime\prime}$
appearing in the last expression above are not necessarily critical.
We remark (as in Theorem 4.2) that the limit
$$\overline \lim_r\frac {1}{T_f(r)}\int_{S(r)}\log {\vert f^\prime \vert ^2_{\omega_{X, B}}}
d\mu_r$$
is non-positive, as it follows from the logarithmic derivative lemma, i.e. this term 
is not affected by the poles of $\omega_{X, B}$. 

In conclusion, the presence of a
log-smooth divisor on $X$ invariant by $\cF$ improves substantially the lower bound we have obtained in
Theorem \ref{monq}, since the main negative terms are defined by $\cJ_{\cF_s, B}$.  
\qed

\medskip 

\subsection{Foliations with reduced singularities on surfaces}

In this subsection we assume that the dimension of $X$ is equal to $n= 2$. As an application of the results in the preceding paragraph, we obtain here a complete analogue of some results originally
due to Michael McQuillan \cite{McQ}.

The twisted vector 
field $V$ defining the foliation $\cF$ is locally given by the expression
$$v_\alpha= a_{\alpha 1}\frac{\partial}{\partial z}+ a_{\alpha 2}\frac{\partial}{\partial w}.$$
In this paragraph we will assume that \emph{the singularities of $\cF$ are reduced}, i.e.
the linearization of the vector field at a singular point has at least a non-zero eigenvalue.
A result by Seidenberg implies that this situation can be achieved after finitely many 
monoidal transformations.
We extract next the following important consequences from the classification theory of foliations with reduced singularities in dimension two \cite{Seiden}.

\begin{enumerate} 

\item[($s_1$)] \emph{A singular point $x_0$ of $\cF$ is called non-degenerate if we have 
$$C^{-1}\leq \frac {|a_{\alpha 1}(z, w)|^2+ |a_{\alpha 2}(z, w)|^2}{|z|^2+ |w|^2}\leq C$$ 
for some coordinate system $(z, w)$ centered at $x_0$.} 
\smallskip

\item[($s_2$)] \emph{A singular point $x_1$ of $\cF$ is called degenerate if we have 
$$C^{-1}\leq \frac {|a_{\alpha 1}(z, w)|^2+ |a_{\alpha 2}(z, w)|^2}{|z|^2+ |w|^{2k}}\leq C$$ 
for some coordinate system $(z, w)$ centered at $x_1$, where $k\geq 2$ is an integer.}
\smallskip

\item[($s_3$)] \emph{Any singular point of $\cF$ is either non-degenerate or degenerate.}

\smallskip

\item[($s_4$)] \emph{For any blow-up $p: \wh X\to X$ of a point $x_{0}$ on the surface $X$ we denote by $\wh \cF:= p^{-1}(\cF)$ the induced foliation on $\wh X$. Then we have 
$$\displaystyle p^\star T_{\cF}= T_{\wh \cF}$$
and moreover, the foliations $\cF$ and $\wh \cF$ have the same 
number of degenerate singular points. In addition, the number $``k"$ appearing in the inequality 
($s_2$) is invariant.}
\end{enumerate}

\medskip

\noindent For a proof of the preceding claims ($s_1$)--($s_3$)
we refer to the paper \cite{Seiden}. As for the property ($s_4$), we can verify it by an explicit computation, 
as follows.

Locally near the point $x_{0}$ the equations of the blow-up map $p$ are given by
$$(x, y)\to (x, xy)\quad \hbox{ or } (x, y)\to (xy, x)$$
corresponding to the two charts covering $\bP^1$. Then the expression of the vector field
defining the foliation $\wh \cF$ on the first chart is as follows
$$a_{\alpha 1}(x, xy)\frac{\partial}{\partial x}+ 
\Big(\frac{a_{\alpha 2}(x, xy)}{x}- y\frac{a_{\alpha 1}(x, xy)}{x}\Big)\frac{\partial}{\partial y}.$$
We denote by $A_{\alpha 1}$ and $A_{\alpha 2}$ the coefficient of $\frac{\partial}{\partial x}$ and 
$\frac{\partial}{\partial y}$ in the expression above, respectively. They are holomorphic functions, and the set of their common zeroes is discrete. Indeed, this is clear if the singularity 
$x_0$ is non-degenerate. If $x_0$ is a degenerate singularity of $\cF$, then this can be
verified by an explicit computation, given the normal form \cite{Seiden}
$$v_\alpha= \big(z(1+\tau w^k)+ wF(z,w)\big)\frac{\partial}{\partial z}+ w^{k+1}\frac{\partial}{\partial w},$$
where $F$ is a function vanishing to order $k$, and $\tau$ is a complex number. In conclusion, the transition functions for the tangent bundle of $\wh\cF$ are the same as the ones 
corresponding to $\cF$, modulo composition with the blow-up map $p$, so our statement is proved.

\subsubsection{Intersection with the tangent bundle}

\smallskip

\noindent In the context of foliations with reduced singularities, the lower bound obtained in Theorem \ref{monq} can be improved substantially, as follows.

\begin{theorem}\label{taut}
Let $(X, \cF)$ be a non-singular compact complex surface, endowed with a foliation. Let
$f:\cY\to X$ be a Zariski dense holomorphic map tangent to $\cF$; here
$(\cY, \sigma)$ is a parabolic Riemann surface 
such that $\mathfrak{X}_\sigma (r)= o(T_f(r))$. If the singularities of $\cF$ are reduced, 
then we have 
$$\int_XT[f]\wedge c_1(T_{\cF})\geq 0.\leqno(39)$$
\end{theorem}
\smallskip

\begin{proof}
If $\cY= \bC$, then Theorem \ref{taut} is one of the key result established in \cite{McQ}. The original arguments in this article can be adapted to the parabolic setting we are interested in,
as we will sketch next.
\smallskip

\noindent $\bullet$ Let $x_0\in \cF_{\rm sing}$ be a singular point of the foliation. We will assume for simplicity that $p$ is non-degenerate; the general case is a little bit more complicated technically, but the main ideas are the same. 
We denote by $\pi: X_1\to X$ the blow-up of $X$ at 
$x_0$, and let $E_1$ be the corresponding exceptional divisor. Let $\cF_1$ be the foliation 
$\pi^\star \cF$ on $X_1$; then $E_1$ is an invariant curve. The foliation $\cF_1$ has two
singularities on $E_1$, say $\wh x_1$ and $\wh y_1$, both non-degenerate. 

We repeat this procedure, and blow-up $\wh x_1$ and $\wh y_1$. On the surface $X_2$ obtained in this way, the inverse image of $x_0$ is equal to 
$$B:= \wh E_1+ E_2+ E_3,$$ where $\wh E_1$ is the proper transform of $E_1$ and $E_2, E_3$ are the
exceptional divisors. In the new configuration, we have 4 singular points of the 
induced foliation $\cF_2$. Two of them belong to $\wh E_1\cap E_2$ and
$\wh E_1\cap E_3$, respectively, and we denote $\wh x_2$ and $\wh y_2$ the other ones. 

We have the injection of sheaves
$$0\to T_{\wh\cF}\to T_{\wh X}\langle B\rangle \leqno(40)$$
that is to say, the tangent bundle of $\cF_2$ is a subsheaf of the logarithmic tangent bundle of 
$(X, B)$. The metric on $T_{\wh\cF}$ induced by the morphism above 
is \emph{non-singular} at each of the four singular points above, and moreover, the image of the lift
of the transcendental curve to $X_2$ do not intersect $\wh E_1\cap E_2$ or
$\wh E_1\cap E_3$. Indeed, these singularities of the foliation $\wh \cF$ have the property 
that both separatrices containing them are algebraic sets. So
if one of these points belong to the image of the curve,
then the curve is automatically contained in $\wh E_1, E_2$ or $E_3$
(since they are leaves of the foliation).

 Therefore, by Theorem \ref{monq, B} we do not have any negative
contribution in ($\star_B$) due to these two points; the only term
we have to understand is
$$\nu^T_1(f, \wh x_2)+ \nu^T_1(f, \wh y_2)\leqno(41)$$
i.e. the truncated counting function corresponding to $\wh x_2$ and
$\wh y_2$. 
 
\smallskip

\noindent $\bullet$ Let $T[f]$ be a Nevanlinna current associated to a Zariski-dense 
parabolic curve on the surface $X$. Let $\pi:\wh X\to X$ be the blow-up of $X$ at $x$; we denote by 
$\wh f$ the lift of $f$. Then there exists $T[\wh f]$ a Nevanlinna current associated to 
$\wh f$ such that 
we have
$$\pi^\star T[f]= T[\wh f]+ \rho [E]\leqno(42)$$
where $\displaystyle \rho= \int_{\wh X}T[\wh f]\wedge c_1(E)$ is a positive number, in general smaller than the 
Lelong number of $T$ at $p$. We remark that we have
$$\rho^2+ \int_{\wh X}\{T[\wh f]\}^2= \int_X\{T[f]\}^2$$
where we denote by $\{T[\wh f]\}$ the cohomology class of $T[\wh f]$. Indeed we have
$\displaystyle \int_{\wh X}\pi^\star \{T[f]\}^2= \int_{\wh X}\big(\{T[\wh f]\}+ \rho c_1(E)\big)^2
= \int_{\wh X}\{T[\wh f]\}^2- \rho^2+ 2\rho \int_{\wh X}\{T[\wh f]\}\wedge c_1(E)$, from which the above equality follows, because $\displaystyle \int_{\wh X}\{T[\wh f]\}\wedge c_1(E)= \rho$.

\smallskip

\noindent $\bullet$ If we iterate the blow-up procedure, the quantities $\rho_j$ we obtain
as in (42) verify 
$$\displaystyle \sum_j\rho_j^2\leq \int_X\{T[f]\}^2,\leqno(43)$$ that is to say, the 
preceding sum is convergent.

\smallskip

\noindent $\bullet$ Let $p\in \cF_{\rm sing}$ be a singular point of the foliation $\cF$. We blow-up
the points $x_2$ and $y_2$ and we obtain $x_3$ and $y_3$, plus two singular points at the intersection of rational curves. 
After iterating $k$ times the blow-up procedure described above, the only negative factor we
have to deal with is 
$$-\nu^T_1(f, \wh x_k)- \nu^T_1(f, \wh y_k)\leqno(44)$$
where $X_k$ is the surface obtained after iterating $k$ times the procedure described above,
$f_k$ is the induced parabolic curve and $\cF_k$ is the induced folia	tion. 
We emphasize that even if the number of singular points of the induced foliation has increased, 
the corresponding negative terms we have to take into account in ($\star_B$) remains the same, i.e. 
the algebraic intersection of the lift of $f$ with the two ``extremal" singular points. By using the notations in (44) above, the convergence of the sum (43) implies that we have
$$\sum_{k\geq 1}{\nu^{T}_1(f_k, \wh x_k)}^2+ {\nu^{T}_1(f_k, \wh y_k)}^2< \infty$$
since $\displaystyle {\nu^{T}_1(f_k, \wh x_k)}= \int_{X_k}T[f_k]\wedge c_1(E_k)$.
Hence the algebraic multiplicity term (44) tends to zero as $k\to\infty$.
\smallskip

\noindent $\bullet$ We therefore have, by using (38)
$$\int_{X_k}T[f_k]\wedge c_1(T_{\cF_k})\geq -\varepsilon_k$$
for $\varepsilon_k\to 0$ as $k\to \infty$. 
Since the singularities of $\cF$ are reduced, we use the property
$s_4$ and we have
$$\int_{X_k}T[f_k]\wedge c_1(T_{\cF_k})= \int_{X_k}T[f_k]\wedge \pi_k^\star c_1(T_{\cF})$$
and this last term is simply 
$$\int_{X}T[f]\wedge c_1(T_{\cF})$$ by the projection formula. Hence it is non-negative.
\end{proof}

\medskip

\noindent The theorem above is particularly interesting when coupled with the following very special case of a result due to
Y. Miyaoka, cf. \cite{Mi}

\begin{theorem}
Let $X$ be a projective surface, whose canonical bundle $K_X$ is big. Let $L\to X$ be a line bundle
such that $\displaystyle H^0\big(X, S^mT_X\otimes L\big)\neq 0$. Then $L$ is pseudo-effective.
\end{theorem}

We apply this result for $L= T_{\cF}^\star$, so we infer that $T^\star_{\cF}$ is pseudo-effective
(indeed, the bundle $T_X\otimes T_{\cF}^\star$ has a non-trivial section).
 
Combined with Theorem \ref{taut}, we obtain 
$$\int_XT[f]\wedge c_1(T_{\cF})=0.\leqno(45)$$

\noindent We derive the following consequence.

\begin{theorem}\label{square}
We consider the data $(X, \cF)$ and $f:\cY\to X$ as in Theorem {\rm \ref{taut}}; in addition, we assume that $K_X$ is big. Then we have
$$\int_X\{T[f]\}^2=0.\leqno(46)$$
If $R$ denotes the diffuse part of $T[f]$, then we have $\displaystyle \int_X\{R\}^2=0.$
In particular, since $T[f]$ is already nef, we infer that $\nu(R, x)= 0$ at each point $x\in X$. 
\end{theorem}

\begin{proof}
If the equality (46) does not hold, we remark that  
the class $\{T[f]\}$ contains a K\"ahler current, i.e. there exists $S\in \{T[f]\}$ such that $S\geq \delta\omega$
for some positive constant $\delta$. Unfortunately we cannot use this representative directly in order to conclude that
the intersection number $\displaystyle \int_XT[f]\wedge c_1(T_{\cF})$ is strictly negative (and obtain in this way a contradiction), because of the possible singularities of $S$ and of the positively defined representatives of $c_1(T_{\cF})$. 

We will proceed therefore in a different manner, and
show next that if the conclusion of our theorem does not hold, then we have
$\displaystyle \int_XT[f]\wedge c_1(T_{\cF})<0$, contradicting the relation (45). 

\noindent We will only discuss here the case $K_X$ ample; the general case (i.e. $K_X$ big) is obtained 
in a similar manner. 

Let $\omega$ be a metric on $X$, such that $\Ricci_\omega\leq -\varepsilon_0\omega$.
We have a sequence of K\"ahler classes $(\alpha_k)_{k\geq 1}$ 
whose limit is $\{T[f]\}$.
\smallskip

\noindent We recall the following particular case of Yau's theorem \cite{Yau}
\begin{theorem}\label{yau}\cite{Yau}
  Let $\gamma$ be a K\"ahler form on a compact complex surface $X$. Let $f$
  be a smooth function on $X$ such that
  $$\int_X\gamma^2= \int_X e^f\gamma^2.$$
Then there exists a K\"ahler form $\omega_1\in \{\alpha\}$ such that the equality
$$\omega_1^2= e^f\gamma^2$$
holds at each point of $X$.
\end{theorem}

\noindent This statement can be reformulated as follows.
Let
  $\alpha$ be a K\"ahler class on a compact complex surface $X$. Let $\omega$ be a K\"ahler 
form on $X$ such that 
$$\int_X\alpha^2= \int_X \omega^2.$$
Then there exists a K\"ahler form $\omega_1\in \alpha$ such that the equality
$$\omega_1^2= \omega^2$$
holds at each point of $X$. Indeed, we take an arbitrary K\"ahler metric
$\gamma$ in the class $\alpha$. There exists a function $f$ such that
$\omega^2= e^f\gamma^2$, and Theorem \ref{yau} shows our claim.

\noindent We infer that there exists a sequence of K\"ahler metrics $(\omega_k)_{k\geq 1}$, such that $\omega_k\in \alpha_k$
for each $k\geq 1$, and such that we have
$$\omega_k^2= \lambda_k\omega^2.$$ 
The sequence of normalizing constants 
$\lambda_k$ is bounded from above, and also from below \emph{away from zero}, 
by our assumption $\displaystyle \int_X\{T[f]\}^2>0$. 

We recall now a basic fact from complex differential geometry.
Let $(E, h)$ be a Hermitian vector bundle, and let $\xi$ be a holomorphic section of 
$E$. Then we have
$$\sqrt{-1}\ddbar\log \vert \xi \vert^2_h\geq -\frac{\langle\Theta_h(E)\xi, \xi\rangle}{|\xi|_h^2}$$
where the curvature form $\Theta_h(E)$ corresponds to the Chern connection of $(E, h)$.
We detail here the computation showing the previous inequality.
$$\dbar \log \vert \xi \vert^2_h= \frac{\langle\xi,  D^\prime\xi\rangle}{|\xi|_h^2}$$ 
where $D^\prime$ is the (1,0) part of the Chern connection.
It follows that 
$$\ddbar\log \vert \xi \vert^2_h= \frac{\langle D^\prime\xi, D^\prime \xi\rangle}{|\xi|_h^2}-
\frac{\langle D^\prime\xi, \xi\rangle\wedge \langle \xi, D^\prime\xi\rangle}{|\xi|_h^4}+
\frac{\langle \dbar D^\prime\xi, \xi\rangle}{|\xi|_h^2}.$$
By Legendre inequality we know that 
$$\displaystyle \sqrt{-1}\frac{\langle D^\prime\xi, D^\prime \xi\rangle}{|\xi|_h^2}-
\sqrt{-1}\frac{\langle D^\prime\xi, \xi\rangle\wedge \langle \xi, D^\prime\xi\rangle}{|\xi|_h^4}\geq 0.
$$ As for the remaining term, we have the equality
$$\dbar D^\prime \xi= -\Theta_h(E)\xi$$
so the inequality above is established.

We apply this for $V$ the section of $T_X\otimes T_{\cF}^\star$ which defined the foliation. The formula above gives
$$\sqrt{-1}\ddbar\log \vert V\vert^2_k\geq \Theta_h(T_{\cF})- 
\frac{\langle \Theta_{\omega_k}(T_X)V, V\rangle}{|V|_k^2}$$
where we denote by $|V|_k$ the norm of $V$ measured with respect to
$\omega_k$ and an arbitrary metric $h$ on $T_{\cF}$.

By considering the wedge product with $\omega_k$ and integrating over $X$ the above inequality we obtain
$$\int_X\omega_k\wedge c_1(T_{\cF})\leq \int_X\frac{\Ricci_{\omega_k}(V, \ol V)}{|V|_{\omega_k}^2}\omega_k^2.$$
We remark that the expression 
$\displaystyle \frac{\Ricci_{\omega_k}(V, \ol V)}{|V|_{\omega_k}^2}$ under the integral sign is well-defined, even if $V$ is only a vector field with values in $T_{\cF}^\star$. Given the Monge-Amp\`ere equation satisfied 
by $\omega_k$ we have
$$\int_X\frac{\Ricci_{\omega_k}(V, \ol V)}{|V|_{\omega_k}^2}\omega_k^2= 
\lambda_k\int_X\frac{\Ricci_{\omega}(V, \ol V)}{|V|_{\omega_k}^2}\omega^2.$$
Since $\Ricci_\omega$ is negative definite, we have
$$\int_X\frac{\Ricci_{\omega}(V, \ol V)}{|V|_{\omega_k}^2}\omega^2\leq
\int_{U_k}\frac{\Ricci_{\omega}(V, \ol V)}{|V|_{\omega_k}^2}\omega^2$$
for any open set $U_k\subset X$.

We have $\displaystyle \int_X\omega_k\wedge \omega\leq C$
for some constant $C$ independent of $k$. So there exists an open set $U_k$
of large volume (with respect to $\omega$), on which the trace of 
$\omega_k$ with respect to $\omega$ is bounded uniformly with respect to $k$, i.e. there exists a constant $C_1$
such that 
$$\omega_k|_{U_k}\leq C_1\omega|_{U_k}$$
for any $k\geq 1$.
With this choice of $U_k$, we will have
$$\int_{X}\alpha_k\wedge c_1(T_{\cF})\leq \int_{U_k}\frac{\Ricci_{\omega}(V, \ol V)}{|V|_{\omega_k}^2}\omega^2\leq C_2
\int_{U_k}\frac{\Ricci_{\omega}(V, \ol V)}{|V|_{\omega}^2}\omega^2< -\delta$$
for some strictly positive quantity $\delta$, and the first part of Theorem \ref{square} is established 
by taking $k\to \infty$.
\smallskip

We write next the Siu decomposition of $T[f]$ 
$$T[f]= \sum_j\tau_{j}[C_j]+ R\leqno(46)$$
and we remark that the class of the current $R$ is nef. Indeed, using Demailly approximation theorem we have $R= \lim R_\varepsilon$, where for each $\varepsilon> 0$ the current 
$R_\varepsilon$ is closed, positive, and non-singular 
in the complement of a finite set of points. Hence $R_\varepsilon^2$ is well-defined and positive as a current; in particular, $\int_X\{R\}^2\geq 0$. 
Then we have $\displaystyle  \int_X\{T[f]\}^2\geq \int_X\{R\}^2$, and it follows that 
$\int_X\{R\}^2=0$.
Thus using a local computation with potential of the current $R$ we infer that 
$$\nu(R, x)= 0\leqno(47)$$
for any $x\in X$, i.e. the Lelong number of $R$ at each point of $X$ is equal to zero.
\end{proof}

\medskip

\subsubsection{Intersection with the normal bundle, {\rm I}}

Our aim in this part is to evaluate the degree of the curve 
$f$ on the normal bundle of the foliation $\cF$. 
To this end we follow
 \cite{Brun1}, \cite{McQ} up to a certain point; their approach 
relies on the Baum-Bott formula \cite{BB}, which is what 
we will survey next. 
\smallskip

\noindent Actually, we will first give the precise expression of a representative of the Chern class
$$c_1\big(N_{\cF, B}\big)\leqno(48)$$
in arbitrary dimension, 
where $B=B_1+\dots+ B_N$ is a divisor on $X$, such that $(X, B)$ verifies the conditions (a)
and (b) in subsection 6.2. The vector bundle in (48) is defined by the exact sequence

$$0\to T_{\cF}\to T_X\langle B\rangle\to N_{\cF, B}\to 0.\leqno (49)$$
We remark that in the case of surfaces, this bundle equals 
$N_{\cF}\otimes \cO(-B)$. 

Let $U_\alpha\subset X$ be an open coordinate set, and let ($z_{1}, \dots, z_{n}$) be a coordinate system on $U_\alpha$. 
We assume that 
$$\Supp(B)\cap U_\alpha= (z_{1}\dots z_{p}=0).$$
Since $B$ is invariant by $\cF$, the vector field giving the local trivialization of $T_{\cF}$
is written as
$$v_\alpha= \sum_{j=1}^pz_jF_{j\alpha}(z)\frac{\partial}{\partial z_j}+ \sum_{i=p+ 1}^n
F_{i\alpha}(z)\frac{\partial}{\partial z_i}\leqno(50)$$
where $F_{j\alpha}$ are holomorphic functions defined on $U_\alpha$.

With respect to the coordinates $(z)$ on $U_\alpha$, the canonical bundle associated to 
$(X, B)$ is locally trivialized by
$$\Omega_\alpha:= \frac{dz_1}{z_1}\wedge \cdots \wedge \frac{dz_p}{z_p}
\wedge dz_{p+1}\wedge \dots\wedge dz_n,\leqno(51)$$
and therefore we obtain a local trivialization of the bundle $\det \big(N_{\cF, B}^\star\big)$ 
by contracting the form (51) with the vector field (50), i.e.
$$\omega_\alpha:= \Lambda_{v_\alpha}(\Omega_\alpha),\leqno(52)$$
that is to say

\begin{equation}
\begin{split}
\omega_\alpha= & \sum_{j=1}^p(-1)^{j+1}F_{j\alpha}(z)\frac{dz_1}{z_1}\wedge \cdots \widehat{\frac{dz_j}{z_j}}\wedge \cdots \wedge \frac{dz_p}{z_p}
\wedge dz_{p+1}\wedge \dots\wedge dz_n \\
+ & \sum_{i=p+1}^n(-1)^{i+1}F_{i\alpha}(z)\frac{dz_1}{z_1}\wedge \cdots 
\wedge \frac{dz_p}{z_p}
\wedge dz_{p+1}\wedge \cdots \wedge\widehat{{dz_i}}\wedge \cdots dz_n
\nonumber
\end{split} 
\end{equation}
We define the differential 1-form $\xi_{\alpha}$ on $U_\alpha$
by the formula
$$\xi_{\alpha, \varepsilon}:= \frac{F^{(1)}_{\alpha}}{\Vert G_\alpha\Vert^2}\sum_j\ol G_{j\alpha}dz_j\leqno(53) $$
where we have used the following notations. If the index $j$ is smaller than $p$, 
then $G_{j\alpha}:= z_j F_{j\alpha}$, and if $j> p$, then we define $G_{j\alpha}:= F_{j\alpha}$. Also, we write
$$F^{(1)}_{\alpha}(z):= \sum_{j=1}^pz_jF_{j\alpha, z_j}(z)+ \sum_{i=p+1}^nF_{i\alpha, z_i}(z)$$
where $\displaystyle F_{k\alpha, z_r}$ is the partial derivative of $F_{k\alpha}$ with respect to $z_r$, and
$$\Vert G_{\alpha}\Vert^2:= \sum_i|G_{i\alpha}|^2.$$
Then we have the equality
$$d\omega_\alpha= \xi_{\alpha}\wedge \omega_\alpha\leqno(54)$$
as one can check by a direct computation.

Let $\alpha, \beta$ be two indexes such that $U_\alpha\cap U_\beta\neq \emptyset$. Then we have
$$\omega_\alpha= g_{\alpha\beta}\omega_{\beta}\leqno(55)$$
on $U_\alpha\cap U_\beta$; here $g_{\alpha\beta}$ are the transition functions for the determinant bundle of $N_{\cF}^\star(B)$. By differentiating the relation (55), we obtain
$$d\omega_\alpha= d\log g_{\alpha\beta}\wedge \omega_\alpha+ g_{\alpha\beta}d\omega_{\beta}\leqno(56)$$
hence from the relation (54) we get
$$\big(\xi_{\alpha}- \xi_{\beta}- d \log g_{\alpha\beta}\big)\wedge \omega_\alpha= 0\leqno(57)$$
on $U_\alpha\cap U_\beta$.
\medskip

\noindent For the rest of this section, we assume that the singularities of $\cF$ are \emph{isolated and 
non-degenerate}.

Then we can find a family of non-singular 
1-forms $\gamma_{\alpha }$ so that we have
$$\xi_{\alpha}- \xi_{\beta}- d \log g_{\alpha\beta}= 
\gamma_{\alpha}- \gamma_{\beta};\leqno(58)$$
for example, we can take 
$$\gamma_{\alpha}:= \sum_\beta \rho_\beta \big(\xi_{\alpha}- \xi_{\beta}- d \log g_{\alpha\beta}\big),\leqno(59)$$
where $(\rho_\alpha)$ is a partition of unit subordinate to the covering $(U_\alpha)$. Here we assume that each 
$U_\alpha$ contains at most one singular point of $\cF$. 

\noindent Then the global 2-form whose restriction to $U_\alpha$ is
$$\Theta|_{U\alpha}:= \frac{1}{2\pi \sqrt{-1}}d(\xi_{\alpha}- \gamma_{\alpha})\leqno(60)$$
represents the first Chern class of $N_{\cF, B}$.
\smallskip

\noindent Let $T$ be a closed positive current of type $(n-1, n-1)$ on $X$; we assume that $T$ is 
\emph{diffuse and directed by the foliation} $\cF$. This implies that for each index $\alpha$
there exists a positive measure $\tau_\alpha$ on $U_\alpha$ such that 
$$T|_{U_\alpha}= \tau_\alpha\omega_\alpha\wedge \overline\omega_\alpha.\leqno(61)$$
We intend to use the representative (60) in order to evaluate the degree of the current $T$ on the 
normal bundle of the foliation. Prior to this, we have to regularize the forms $\xi_{\alpha}$. 
This is done simply by replacing $\xi_{\alpha}$ with $h_\alpha\xi_{\alpha}$, where 
$h_\alpha$ is equal to 0 in a open set containing the singularity of $\cF$, and it is equal to one 
out of a slightly bigger open set. A specific choice of such functions will be made shortly; we remark that the effect of the multiplication with $h_\alpha$ is that the equality (54) is only verified 
in the complement of an open set containing the singular point.
The formula for the curvature becomes $\displaystyle \Theta|_{U\alpha}:= \frac{1}{2\pi \sqrt{-1}}d(h_\alpha\xi_{\alpha}- \gamma_{\alpha})$.
\medskip

\noindent As a consequence, we have the next statement, which is the main result of this subsection.

\begin{lemma}
Let $\cF$ be a foliation by curves with isolated singularities on $X$. We consider a closed positive current $T$ of bidimension (1,1), which is directed by the foliation $\cF$. If 
$\displaystyle \cF_{\rm sing}\cap \Supp(T)$ consists in non-degenerate points only, then we have
$$\int_{X}T\wedge \Theta\geq -C\sup_{x\in \cF_{\rm sing}}\nu(T, x)\leqno(62)$$
for some positive constant $C$.
\end{lemma}

\begin{proof}
We first observe that we have 
$$d\gamma_{\alpha }\wedge T= 0.\leqno(63)$$
Indeed, $\omega_\alpha$ is holomorphic and $\gamma_\alpha$
is a (1,0) form, so we have
$$d\gamma_{\alpha}\wedge T= \tau_\alpha
\overline\partial(\gamma_{\alpha}\wedge \omega_\alpha)\wedge \overline \omega_\alpha.\leqno(64)$$
The equality (63) is therefore a consequence of (57)-(59).

In order to compute the other term of (62), we use the relation (54), and we infer
that we have
$$\overline\partial (h_\alpha \xi_{\alpha})\wedge \omega_\alpha= 
\overline\partial h_\alpha\wedge\xi_{\alpha}\wedge \omega_\alpha\leqno(65)$$
for any choice of $h_\alpha$.

In what follows, we will drop the index $\alpha$, and concentrate around a single singular point
namely the origin of the coordinate system.
Let $\chi$ be a smooth function vanishing near zero, and which equals 1 for $x\geq 1/2$. 
We consider 
$$h_r(z):= \chi\big(\Vert z\Vert^r\big)\leqno(66)$$
for each $r> 0$. We have then $0\leq h_r\leq 1$, and as $r\to 0$, $h_r$ converges to the 
the characteristic function of $U\setminus 0$.

\noindent If moreover the singular point $z= 0$ is non-degenerate, in the sense that we have
$$C^{-1}< \frac{\Vert G(z)\Vert}{\Vert z\Vert}< C$$
for some constant $C> 0$, then we have to bound from above the mass of the measure 
$$r\frac{\chi^\prime\big(\Vert z\Vert^r\big)}{\Vert z\Vert^{2-r}}\sigma_T$$ 
as $r\to 0$; here $\sigma_T$ is the trace measure of $T$.
In other words we have to obtain an upper bound for the integral
$$r\int_{C_1< \Vert z\Vert^r< C_2}\Vert z\Vert^{r-2}\sigma_T.\leqno(67)$$
We observe that we have
$$\int_{C_1< \Vert z\Vert^r< C_2}\!\!\!\Vert z\Vert^{r-2}\sigma_T= \int_0^\infty\sigma_T
\Big(\big(\Vert z\Vert^{r-2}> s\big)\cap (C_1< \Vert z\Vert^r< C_2)\Big)ds;
$$
and up to a $\cO(r)$ term, the quantity we have to evaluate is smaller than
$$r\int_1^{C_1^{-\frac{2-r}{r}}}\sigma_T\big(\Vert z\Vert< s^{\frac{1}{r-2}}\big)ds.\leqno(68)$$
If $\nu$ is  
Lelong number of $T$ at $0$, then we have $\sigma_T\big(\Vert z\Vert< \tau\big)\leq 2\nu\tau^2$,
as soon as $\tau$ is small enough.
Therefore, a quick computation shows that the integral (68) is bounded by $C\nu$, for some positive constant $C> 0$.
\end{proof}
\smallskip

As one can see from the proof, this kind of arguments can only be used if
the singular point is non-degenerate. However, this is sufficient for our purposes,
given the next result.

\smallskip

\begin{theorem}\label{supp}
Let $X$ be a projective surface endowed with a holomorphic foliation $\cF$.
We assume that the singularities of $\cF$ are reduced. 
Let $R$ be a diffuse closed positive (1,1) current on $X$, directed by $\cF$. 
Let $\{x_1,\dots x_d\}$ be the set of degenerate singularities of $\cF$. Then we have
$$\displaystyle \Supp(R)\subset X\setminus \{x_1,\dots x_d\}.$$
\end{theorem}

\begin{proof} Suppose $(0,0)$ is a degenerate reduced singularity of $\cF$. It is well-known that
the holonomy map $h$ at such a point is tangent to identity. We assume that the separatrix is given by the 
$z$-axis, and that the 
$w$-axis is transversal to $\cF$. 

The argument is based on the fact that 
the dynamics of such a map near the origin is well-understood.
It is the L\'eau ``flower theorem''. There is a neighborhood of zero, say $U$,
such that 
$$U\setminus 0=  \bigcup_{1\leq j\leq p}P_{j}^+\cup P_{j}^-.$$
If the point $z$ belongs to the petals $P_{j}^+$, then
we have $h^{n}(z)\to 0$, where we denote by $h^n$ the $n^{\rm th}$ iteration of
$h$. If  $z$ belongs to the petals $P_{j}^-$, then
we have $h^{-n}(z)\to 0$. 

Suppose that the current $R$ 
has mass in a small open $\Xi$ set near 0, which can be assumed to have the form
$\Xi= P_{1}^+\times \Delta$, where $\Delta$ is a disk.
Let $x_{0}:= (0, w_{0})$ be a point in $U\cap\Supp(R)$, such that
$w_{0}\neq 0$. 

The current $R$ can be written on $\Xi$ as 
$$R= \int_{P_{1}^+}[V_{w}]d\mu(w),$$
where $V_{w}$ is the plaque through $w$.

We show now that $\mu$ has no mass out of the origin.
Let $W\subset P_{1}^+$ be an open set containing $w_{0}$, such that 
$W\cap (w=0)= \emptyset$. We construct a new open set 
$W_{1}\subset P_{1}^+$ via the holonomy map, as follows. 
Without loss of generality we can assume that the 
petal $P_{1}^+$ is invariant by $h$. We define $W_{1}:= h^{n_{1}}(W)$
where $n_1$ is large enough in order to have 
$W_{1}\cap W= \emptyset$.

We can restart this procedure with $W_{1}$, obtaining $W_{2}, \dots, W_{k}$. In conclusion, 
we get a sequence of open sets
$(W_{k})\subset P_{1}^+$ such that $W_{i}\cap W_{j}= \emptyset$ if $i\neq j$. The mass
of each set with respect to the transverse measure is preserved,
since the measure is invariant by the holonomy map. This is equivalent to the fact that 
the current $R$ is closed and directed by $\cF$.
Therefore we obtain a contradiction.
\end{proof}

\smallskip

\noindent In conclusion, we obtain the next result.

\begin{theorem} \label{lelong}
Let $X$ be a projective surface endowed with a holomorphic foliation $\cF$.
We assume that the singularities of $\cF$ are reduced. 
Let $R$ be a diffuse closed positive (1,1) current on $X$, directed by $\cF$. Then there exists a constant $C> 0$ such that
$$\int_XR\wedge c_1\big(N_{\cF}(B)\big)\geq -C\sup_{x\in \cF_{\rm sing}}\nu(R, x).\leqno(69)$$
In particular, if we have $\nu(R, x)= 0$ for every $x\in \cF_{\rm sing}$, then we have 
$\int_XR\wedge c_1\big(N_{\cF}(B)\big)\geq 0$.
\end{theorem}

\begin{proof}
By Theorem \ref{supp}, $R$ has no mass near degenerate points.
In the case of a non-degenerate singularity, we use Lemma 6.12, since $\Theta$ is a representative
of $c_{1}(N_{\cF}(-B))$.
The theorem is thus proved.
\end{proof}

\medskip

\noindent As a conclusion of the considerations in the preceding subsections, we infer the following
``parabolic version" of the result in \cite{McQ}.

\begin{corollary}\label{c6}
Let $(X, \cF)$ be a surface of general type endowed with a holomorphic foliation $\cF$, and let $f: \cY\to X$ be a holomorphic curve, where $\cY$ is a parabolic Riemann surface such that $\mathfrak{X}_\sigma (r)= o(T_f(r))$.
If $f$ is tangent to $\cF$, then the dimension of the Zariski closure of the image of $f$ is at most 1.
\end{corollary}

\begin{proof}
The first thing to remark is that the hypothesis is stable by 
blow-up, hence we can assume that the singularities of $\cF$ are reduced. We can also assume that the current associated to $f$ can be written as
$$T[f]= \sum_j\nu_j[C_j]+ R$$
where $C:= C_1+\dots C_N$ verifies (a) and (b), and the Lelong numbers of
$R$ may be positive eventually at a finite subset of $X$. We argue by contradiction, so we add the hypothesis that the image of $f$ is Zariski dense.

If so, by Theorem \ref{taut} and Theorem \ref{square} we infer that we have $\displaystyle \int_XT[f]\wedge c_1(T_{\cF})= 0$, as well as 
$\displaystyle \int_X\{T[f]\}^2=0$ (we remark that the fact that the image of $f$ is Zariski dense is essential here).
This implies that the Lelong numbers of the diffuse part $R$ of $T[f]$ are equal to zero. By Theorem
\ref{lelong}, we infer that 
$$\int_XR\wedge c_1(N_{\cF}(-C))\geq 0.$$
The next step is to show the inequality
$$\int_XT[f]\wedge c_1(N_{\cF})\geq \int_XR\wedge c_1(N_{\cF}(-C)).$$
We follow the presentation in \cite{Brun1}, and we first show that 
we have $\displaystyle \int_X(T[f]-R)\wedge c_1(N_{\cF})\geq \int_X(T[f]-R)\wedge C.$ We recall the formula
$$\int_{C_j}c_1(N_{\cF})= \int_Xc_1(C_j)^2+ Z(C_j, \cF)$$
where $Z(C_j, \cF)$ is the multiplicity of the singularities of $\cF$ along the curve $C_j$.
We clearly have 
$$Z(C_j, \cF)\geq \sum_{j\neq k}c_1(C_j)\wedge c_1(C_k)$$
and thus we obtain
$$\sum_j\int_X\nu_jc_1(C_j)\wedge c_1(N_{\cF})\geq \sum_{j, k}\int_X\nu_jc_1(C_j)\wedge
c_1(C_k).$$
The preceding inequality implies that 
$$\int_XT[f]\wedge c_1(N_{\cF})\geq \int_XR\wedge c_1(N_{\cF})+ \int_X(T[f]- R)\wedge
c_1(C)$$
and given that $T[f]$ is nef, we infer that 
$$\int_XR\wedge c_1(N_{\cF})+ \int_X(T[f]- R)\wedge
c_1(C)\geq \int_XR\wedge c_1(N_{\cF}(-C)).$$
As a consequence we obtain 
$$\int_XT[f]\wedge c_1(N_{\cF})\geq 0.$$
Since $\displaystyle c_1(N_{\cF})+ c_1(T_{\cF})= c_1(X)$, 
the inequalities above imply that $\int_XT[f]\wedge c_1(K_X)\leq 0$. This is absurd since
$K_X$ is big and $T[f]$ is nef.
\end{proof}
\bigskip

\section{Brunella index theorem}

\noindent We will survey here a large part of the proof of 
the following beautiful result, due to M.~Brunella 
\cite{Brun1}. The context is as follows: $X$ is a compact complex surface, endowed with a holomorphic foliation $\cF$ and $f$ is a parabolic curve tangent to $\cF$. 

\begin{theorem} \label{index}
{\rm \cite{Brun1}}  
If the curve $f:\cY\to X$ has a Zariski dense image, then we have
$$\int_XT[f]\wedge c_1(N_{\cF})\geq 0.\leqno(70)$$
\end{theorem}

\noindent We remark that a similar statement 
appeared 
in the preceding subsection, but it was obtained in a very indirect way,
under the assumption that the canonical bundle of $X$ is big
(and as part of an argument by contradiction assuming that the image of
$f$ is Zariski dense...). 

M. Brunella's arguments involve many considerations from dynamics, including a 
study at Siegel points of $\cF_{\rm sing}$.
They seem difficult to 
accommodate to higher dimensional case; in the next lemma we treat the Siegel points by using 
the following elegant lemma for which we are indebted to M. McQuillan, cf.\ \cite{MQ}.

\begin{lemma}\cite{MQ}\label{trivial}
Let $T$ be a closed positive current on a surface $X$. We assume that 
$T$ is diffuse and directed by a foliation $\cF$. Let 
$x_0\in \cF_{\rm sing}$ be a singular point of $\cF$. We assume that there exists a coordinate open set 
$U\subset X$ centered at $x_0$ and local coordinates $(z, w)$ on $U$ such that the restriction
$\cF|_U$ is given by $adw- bdz= 0$, where 
$$
a(z, w)= z(1+ o(z, w)), \quad b(z, w)= \lambda w(1+ o(z, w))
$$
where $\lambda\in \bR$ is a strictly negative real number.
Then $\nu(T, x_0)= 0$, i.e. the Lelong number of $T$ at $x_0$ vanishes.
\end{lemma}

\begin{proof} Let $\ep> 0$ be a positive number, and let $\Omega_\ep$ be the open set 
$(|z|< \ep)\times (|w|<\ep)\subset \bC^2$. We denote by $\partial \Omega_\ep$ the boundary of this domain.
We proceed in two steps, as follows.
\smallskip

\noindent {\rm (I)} There exists a constant $C> 0$ such that 
$$\nu(T, x_0)\leq C\int_{\partial\Omega_\ep}d^c\log |zw|^2\wedge T.$$

\noindent (II) We have $\displaystyle \lim_{\ep\to 0}\int_{\partial\Omega_\ep}d^c\log |zw|^2\wedge T= 0.$
\smallskip

\noindent The lemma follows from the two statements (I) and (II) above; we will show that they hold true in 
what follows. 
   
\noindent {\sl Proof of }(I). Let $\ep> \delta> 0$ be two positive real numbers; we have 
\begin{equation} 
\begin{split}
\nu(T, x_0)\leq & \int_{|z|^2+ |w|^2< \delta^2}dd^c\log\left(|z|^2+ |w|^2\right)\wedge T \\ 
= & \frac{1}{\delta^2}
\int_{|z|^2+ |w|^2< \delta^2}dd^c\left(|z|^2+ |w|^2\right)\wedge T\\
\leq & \frac{1}{\delta^2}\int_{|z|< \delta,  |w|< \delta}dd^c\left(|z|^2+ |w|^2\right)\wedge T\\
= & \frac{1}{\delta^2}\left(\int_{|z|= \delta,  |w|< \delta}d^c |z|^2\wedge T+ 
\int_{|z|< \delta,  |w|= \delta}d^c |w|^2\wedge T\right)\\
+ & \frac{1}{\delta^2}\left(\int_{|z|= \delta,  |w|< \delta}d^c |w|^2\wedge T+ 
\int_{|z|< \delta,  |w|= \delta}d^c |z|^2\wedge T\right).
\end{split}
\nonumber
\end{equation}

\noindent By using the relation 
$$\displaystyle d^c|w|^2\wedge T= |w|^2(\lambda^{-1}+ o(z,w))d^c\log|z|^2\wedge T\leqno(\star)$$ 
(which is a consequence of the 
hypothesis), we have 
$$\int_{|z|= \delta,  |w|< \delta}d^c |w|^2\wedge T
= \int_{|z|= \delta,  |w|< \delta}|w|^2(\lambda^{-1}+ o(z,w))d^c\log|z|^2\wedge T$$
so in particular the limit as $\delta\to 0$ of the term
$$\frac{1}{\delta^2}\left(\int_{|z|= \delta,  |w|< \delta}d^c |w|^2\wedge T+ 
\int_{|z|< \delta,  |w|= \delta}d^c |z|^2\wedge T\right)
$$
exists, and it is bounded by the limit of
$$I_\delta:= \frac{1}{\delta^2}\left(\int_{|z|= \delta,  |w|< \delta}d^c |z|^2\wedge T+ 
\int_{|z|< \delta,  |w|= \delta}d^c |w|^2\wedge T\right)$$
as $\delta\to 0$. We have 
\begin{equation} 
\begin{split}
I_\delta= & \int_{|z|= \delta,  |w|< \delta}d^c \log|z|^2\wedge T+ \int_{|z|< \delta,  |w|= \delta}d^c \log|w|^2\wedge T\\
\leq & \int_{|z|= \delta,  |w|< \ep}d^c \log|z|^2\wedge T+ \int_{|z|< \ep,  |w|= \delta}d^c \log|w|^2\wedge T\\
:= & I_\delta(\ep, 1)+ I_\delta(\ep, 2).
\end{split}
\nonumber
\end{equation}
By Stokes theorem we obtain
\begin{equation} 
I_\delta(\ep, 1)= \int_{|z|= \ep,  |w|< \ep}d^c \log|z|^2\wedge T
+ \int_{\delta< |z|<\ep,  |w|= \ep}d^c \log|z|^2\wedge T
\nonumber
\end{equation}
and thus we infer that 
\begin{equation} 
\lim_{\delta\to 0}\left(I_\delta(\ep, 1)+ I_\delta(\ep, 2)\right)\leq C\int_{\partial\Omega_\ep}d^c\log |zw|^2\wedge T.
\nonumber
\end{equation}
The proof of the point (I) is finished by combining the relations above.
\smallskip

\noindent {\sl Proof of }(II). We use the relation 
\begin{equation} 
\frac{dw}{w}\wedge T= (\lambda+ O(z,w))\frac{dz}{z}\wedge T
\nonumber
\end{equation}
and we obtain 
\begin{equation} 
\int_{\partial \Omega_\ep}d^c\log|w|^2\wedge T= \lambda \int_{\partial \Omega_\ep}d^c\log|z|^2\wedge T+ O(\ep)
\nonumber
\end{equation}
so we have 
\begin{equation} 
\int_{\partial \Omega_\ep}d^c\log|zw|^2\wedge T=  \lambda\int_{\partial \Omega_\ep}d^c\log|z|^2\wedge T+ 
\lambda^{-1}\int_{\partial \Omega_\ep}d^c\log|w|^2\wedge T+ O(\ep).
\nonumber
\end{equation}
The conclusion follows, since $\lambda$ is a negative real number and the 
left hand side term of the preceding equality is positive.
\end{proof}
\medskip

\noindent For the rest of the proof of Theorem \ref{index} we refer 
to \cite{Brun1}: it consists in a case-by-case analysis, as follows. By the proof of 
Corollary \ref{c6} it is enough to show that we have 
$$\int_XR\wedge c_1\left(N_{\cF}(-B)\right)\geq 0,$$
where $R$ is the diffuse part of $T[f]$. Theorem \ref{lelong} shows that this 
quantity is bounded from below by the Lelong numbers of $R$ at the singular points of 
the foliation $\cF$. If $x_0\in \cF_{\rm sing}$ is a linearizable singularity, then his 
contribution to the quantity $\displaystyle 
\int_XR\wedge c_1\left(N_{\cF}(B)\right)$ can be seen to be positive by 
a direct computation. If $x_0$ is degenerate, then we have nothing to do, 
thanks to Theorem \ref{supp}. The remaining case follows from Lemma \ref{trivial}:
since the Lelong number is equal to zero, the result follows.\qed 

\medskip
\noindent
We will establish here a generalization of of Theorem \ref{index} in the context of symmetric differentials. 
\begin{corollary}
Let $L$ be a line bundle on $X$, such that 
$$H^0(X, S^mT^\star_X\otimes L)\neq 0.$$
Then we have
$$\int_XT[f]\wedge c_1(L)\geq 0.\leqno(71)$$
\end{corollary}

\begin{proof} Observe that Brunella's result corresponds to the case $m=1$ and $L= N_{\cF}$.
Indeed, if the inequality (71) above does not hold, then given a section $u$ of 
$S^mT^\star_X\otimes L$ we have $u(f^{\prime \otimes m})\equiv 0$, by the vanishing theorem.
Let $\Gamma\subset \bP(T_X)$ be the set of zeros of $u$, interpreted as a section of 
the bundle $\cO(m)\otimes \pi^\star (L)$, where $\pi: \bP(T_X)\to X$ is the natural projection.
The natural lift of the curve $f$ is contained in one of the irreducible components of $\Gamma$;
let $X_m$ be a desingularization of the component containing the image of $f$.

In conclusion, we have a projective surface $X_m$, endowed with a holomorphic foliation 
$\cF_m$, a generically finite morphism $p: X_m\to X$, and a curve $f_m: \cY\to X_m$, such that
$f= p \circ f_m$.

Let $u_m:= p^\star(u)$ be the section of $\displaystyle S^mT^\star_{X_m}\otimes p^\star (L)$
obtained by taking the inverse image of $u$.
We decompose the set $Z_m: =(u_m=0)\subset \bP(T_{X_m})$ as follows
$$Z_m= \sum_j m_j\Gamma_j\leqno(72)$$
where $\Gamma_j\subset \bP(T_{X_m})$ are irreducible hypersurfaces, and $m_j$ are positive integers. 

The image of the lift 
of $f_m$ to $\bP(T_{X_m})$ is contained in $Z_m$, and it is equally contained in the graph
of $\cF_m$. Since the curve $f_m$ is supposed to be Zariski dense, its lift can be contained in 
at most one hypersurface $\Gamma_j$.
We will henceforth assume that the graph of the foliation $\cF_m$ coincides with $\Gamma_1$.

Numerically, we have
$$\Gamma_j\equiv \nu_j\cO(1)+ \pi_m^\star(L_j)\leqno(73)$$
where $\pi_m: \bP(T_{X_m})\to X_m$ is the projection, $\nu_j$ are positive integers, and $L_j$ are line bundles on $X_m$. This is a consequence of the structure of the Picard group of 
$\bP(T_{X_m})$.

The relations (72), (73) show that we have
$$p^\star (L)\equiv \sum m_j L_j.$$
Since the lift of the curve $f_m$ is not contained in any $\Gamma_j$ for $j\geq 2$, we have
$\displaystyle \int_XT[f_m^\prime]\wedge c_1(\Gamma_j)\geq 0$. The tautological inequality
\ref{trivia}
states that $\displaystyle \int_XT[f_m^\prime]\wedge c_1\big({\cO}(1)\big)\leq 0$ 
so we obtain from (73)
$$\int_{X_m}T[f_m]\wedge c_1(L_j)\geq 0\leqno(74)$$
for each index $j\geq 2$. 

But we have assumed that (71) does not hold, so we infer that
$$\int_{X_m}T[f_m]\wedge c_1(L_1)< 0\leqno(75)$$ 
which contradicts Theorem \ref{index} (because $L_{1}$ is the normal bundle of $\cF_{m}$).
\end{proof}
\smallskip

\noindent In particular, this gives a proof of the Green-Griffiths conjecture for minimal general type surfaces 
$X$ such that $c_1^2(X)> c_2(X)$ without using any consideration about the tangent bundles of foliations! Indeed, under this hypothesis we have 
$$H^0(X, S^mT^\star_X\otimes A^{-1})\neq 0$$
where $A$ is ample, and $m\gg 0$ is a positive integer (by a result due to F. Bogomolov, \cite{Bog}; it is at this point that we are using the hypothesis about the Chern classes of $X$). 
So if $f$ is Zariski dense, we obtain a contradiction by the 
corollary above, given the strict positivity of $A$. 
\smallskip

\begin{rem}{\rm
As we see, the proof is by contradiction. However,
by Bogomolov \cite{BoGo} theorem there are only a finite number of 
elliptic or rational curves if $c_1^2> c_2$. Thus there exists a fixed sub-variety of $X$ containing all the images of $\bC$. }
\end{rem}  

\noindent 
Given the aforementioned consequence of Brunella's theorem, it is very tempting to formulate the following problem.

\begin{conjecture}
Let $k, m$ be two positive integers, and let $L\to X$ be a line bundle, such that
$$H^0\big(X, E_{k, m}T^\star_X\otimes L\big)\neq 0.$$
If $f: \cY\to X$ is a Zariski dense parabolic curve such that $\mathfrak{X}_\sigma (r)= o(T_f(r))$, then we have
$$\int_XT[f]\wedge c_1(L)\geq 0.\leqno(76)$$
\end{conjecture}

\noindent This would imply the Green-Griffiths conjecture for surfaces of general type, since in this case it is known that for any ample line bundle $A$ there exists $k\gg m\gg0$ such that 
$H^0\big(X, E_{k, m}T^\star_X\otimes A^{-1}\big)\neq 0$.


\bigskip



\end{document}